\documentclass[reqno]{amsart}

\usepackage{amsthm,amsmath,amssymb}
\usepackage{mathrsfs}
\usepackage{dsfont}
\usepackage{hyperref}
\usepackage{cite}
\usepackage{graphicx}
\graphicspath{ {./Figures/} }
\usepackage{booktabs}

\title{Aggregation Methods for Computing Steady-States in Statistical Physics}

\newcommand{\ignore}[1]{}

\renewcommand{\P}{\mathbb{P}}

\newcommand{\1}{\mathds{1}}
\renewcommand{\t}{\mathrm{t}}
\newcommand{\Real}{\mathbb{R}}

\newcommand{\eps}{\varepsilon}
\DeclareMathOperator{\rg}{Rg}
\DeclareMathOperator{\diag}{diag}
\DeclareMathOperator{\spn}{span}

\newcommand{\gea}[1]{}

\theoremstyle{Theorem}
\newtheorem{theorem}{Theorem}
\theoremstyle{Lemma}
\newtheorem{lemma}{Lemma}
\theoremstyle{Corollary}
\newtheorem{corollary}{Corollary}

\theoremstyle{Remark}

\theoremstyle{Definition}
\newtheorem{definition}{Definition}
\newtheorem{assumption}{Assumption}
\newtheorem{algorithm}{Algorithm}

\author{Gabriel Earle}
\address[Gabriel Earle]{Department of Mathematics and Statistics, University of Massachusetts, Amherst}

\author{Brian Van Koten}
\address[Brian Van Koten]{Department of Mathematics and Statistics, University of Massachusetts, Amherst}
\email[Correspoding author]{bvankoten@umass.edu}
\thanks{BvK and GE were supported by NSF DMS-2012207}

\begin{document}
\maketitle

\begin{abstract}
We give a new proof of local convergence of a multigrid method called iterative aggregation/disaggregation (IAD) for computing steady-states of Markov chains. Our proof leads naturally to precise and interpretable estimates of the asymptotic rate of convergence. We study IAD as a model of more complex methods from statistical physics for computing nonequilibrium steady-states, such as the nonequilibrium umbrella sampling method of Warmflash, et al. We explain why it may be possible to use methods like IAD to efficiently calculate steady-states of processes in statistical physics and how to choose parameters to optimize efficiency.
\end{abstract}

\section{Introduction}

We prove local convergence of iterative aggregation/disaggregation (IAD) for computing steady-states of Markov chains, and we estimate the asymptotic rate of convergence. IAD was devised in the 1960's to solve economic input-output models; see the references given in~\cite{mandel_local_1983,vakhutinsky_iterative_1979}. Substantially equivalent methods were independently developed in the 1980's to calculate steady-states of Markov chains~\cite{koury_iterative_1984,chatelin_acceleration_1982,haviv_aggregationdisaggregation_1987,cao_iterative_1985}. In the 2000's, similar ideas arose for a third time as part of more complex methods for calculating nonequilibrium steady-states and reaction rates in statistical physics~\cite{warmflash_umbrella_2007,vanden-eijnden_exact_2009,bhatt_steady-state_2010,bello-rivas_exact_2015,dinner_trajectory_2018,earle_convergence_2022}. We study IAD as a simple model of these complex methods. We explain why it may be possible to use methods like IAD to efficiently calculate steady-states in statistical physics and how to choose parameters to optimize efficiency. We hope others will apply our results to understand IAD in other contexts.

We call a Markov process \emph{nonequilibrium} if it is irreversible. A physical system subject to nonconservative forces or external flows of energy and matter would typically be modeled by a nonequilibrium process, e.g.\@ a single-molecule experiment where a protein is subjected to a flow of ions~\cite{li_models_2009,dickson_flow-dependent_2011}.
In principle, to sample the steady-state distribution of any ergodic process, reversible or irreversible, one can take the average over a long trajectory. In practice, however, trajectory averages converge to the steady-state very slowly when obstacles like bottlenecks inhibit exploration of the state space. For example, a process modeling a protein may spend most of the time vibrating around some stable folded state, undergoing transitions between different folded states only rarely. Such a process is said to be \emph{metastable}. 

Computing the steady-state of a metastable, nonequilibrium process is especially difficult. 
Reliable methods have been devised to efficiently compute steady-states of reversible, metastable processes, e.g.\@ parallel tempering~\cite{swendsen_replica_1986,geyer_markov_1991}, umbrella sampling~\cite{torrie_nonphysical_1977,kumar_weighted_1992,shirts_statistically_2008}, metadynamics~\cite{laio_escaping_2002}, and adaptive biasing~\cite{darve_calculating_2001}. These methods are essential tools for simulating systems in equilibrium. Unfortunately, however, none of them can compute nonequilibrium steady-states. Each requires either reversibility or knowledge of the steady-state density, and when computing nonequilibrium steady-states one typically knows only the generator of the process. By constrast, in equilibrium, the steady-state has the Boltzmann density, which can almost always be calculated up to a normalizing constant.

Recently, analogous methods have been devised to compute nonequilibrium steady-states and dynamical quantities such as reaction rates. We consider one class derived from umbrella sampling, including nonequilibrium umbrella sampling (NEUS)~\cite{warmflash_umbrella_2007}, trajectory parallelization and tilting~\cite{vanden-eijnden_exact_2009}, weighted ensemble with a direct solve~\cite{bhatt_steady-state_2010}, exact milestoning~\cite{bello-rivas_exact_2015}, trajectory stratification~\cite{dinner_trajectory_2018}, and injection measures~\cite{earle_convergence_2022}. The exact objectives and details of these methods differ significantly, but they are all essentially stochastic evolving particle systems that approximate IAD (or a similar deterministic dynamics) in the limit of a large number of particles. We refer to~\cite{earle_convergence_2022} for details. We ask whether approximating IAD is a suitable goal for an algorithm designed to compute steady-states in statistical physics.

IAD is like an algebraic multigrid method, but it is nonlinear and nonsymmetric which significantly complicates its analysis, cf.\@ Appendix~\ref{apx: iad as multigrid}. In the earliest convergence analysis of IAD known to us, Mandel and Sekerka proved local convergence for a class of problems including the solution of input-output models but not steady-states of Markov chains~\cite{mandel_local_1983}. Later work verified local convergence for Markov chains under various conditions and for various versions of the IAD algorithm~\cite{marek_local_1994,krieger_two-level_1995,marek_convergence_1998,marek_convergence_2003,marek_note_2006,earle_convergence_2022}. We are not aware of a proof of global convergence that holds under general conditions. However, see~\cite{marek_note_2006} for a proof of global convergence under somewhat restrictive conditions and examples where IAD fails to converge. The efficiency of IAD has been studied in special cases, including nearly completely decomposable chains~\cite{koury_iterative_1984} and cyclic chains~\cite{pultarova_fourier_2013}. 

We contribute a new proof of local convergence of IAD under weak conditions that are easy to verify. The usual conditions that guarantee convergence of a Markov chain to a steady-state, irreducibility and aperiodicity, do not suffice to prove even local convergence of IAD, cf.\@ Appendix~\ref{apx: pathological examples} and~\cite{marek_note_2006}. If $P$ is the transition matrix of the chain, we prove local convergence when $P$ and $P^\t P$ are irreducible, cf.\@ Theorem~\ref{thm: local convergence}. It is equivalent to assume that the chain is strictly contracting in a certain norm, cf.\@~Lemma~\ref{lem: convergence of power method}. A sufficient condition is that $P$ be irreducible and have positive diagonal.
 
Our proof of local convergence leads to precise and interpretable estimates of the asymptotic rate of convergence, cf.\@ Theorem~\ref{thm: estimate of rate of convergence with spectrum and angle involved} and Corollary~\ref{cor: estimate of rate in reversible case}. Based on these estimates, we have developed some general advice to guide the choice of parameters in IAD; see the discussion following Corollary~\ref{cor: estimate of rate in reversible case}. We apply our theory of the rate of convergence in Section~\ref{sec: examples} to explain why it may be possible to use methods like IAD to efficiently compute steady-states of reversible or irreversible processes arising in statistical physics and computational chemistry. To be precise, in Section~\ref{sec: examples}, we introduce a family of Markov chains analogous to the metastable diffusion processes that are widely used as models of molecular systems. We then explain why IAD can sometimes efficiently calculate the steady-states of such chains and how to choose the parameters in practice. We illustrate our conclusions with numerical experiments. We developed our rate estimates with these examples from statistical physics in mind, but our results are general, and we hope others will apply them to understand IAD in other contexts.

\section{Notation}

Here, we summarize our notation. Notation is also explained below when it first appears.

\begin{itemize}
  \item $P \in \Real^{N \times N}$ will be an irreducible, column stochastic transition matrix with invariant distribution $\mu \in \Real^N$. 
  \item 
    For any matrix or vector $M$, $M \geq 0$ means all entries of $M$ are nonnegative. $M>0$ means all entries are positive.
  \item $\1$ will denote a vector of all ones and $I$ will denote the identity matrix. The dimension of $\1$ or $I$ will be determined by the context. 
    \item For any $A \subset \{1, \dots, N\}$, $\1_A$ will denote the characteristic function of $A$. That is, $\1_A(x) = 1$ if $x \in A$ and $\1_A(x) = 0$ if $x \notin A$. 
  \item For any vector $\nu \in \Real^k$, we let $\diag(\nu) \in \Real^{k \times k}$ denote the diagonal matrix with $\diag(\nu)_{ii} = \nu_i$ for $i = 1,\dots, k$.
  \item $\lVert \cdot \rVert_\nu$ and $\langle, \rangle_\nu$ denote the $\ell^2(\nu)$-norm and inner product, respectively; see Definition~\ref{def: weighted l2}. For $M \in \Real^{k \times k}$, $M^{\ast,\nu}$ denotes the adjoint of $M$ with respect to the $\ell^2(\nu)$-inner product. In some proofs, we simplify notation, letting $\lVert \cdot \rVert$, $\langle, \rangle$, and $M^\ast$ denote the $\ell^2(1/\mu)$-norm, inner product, and adjoint, respectively. 
  \item For any operator $M \in \Real^{M \times M}$, $\rg(M)$ denotes the range of $M$.
\end{itemize}

\section{The Iterative Aggregation/Disaggregation Method}
\label{sec: IAD}

Iterative Aggregation/Disaggregation (IAD) is a numerical method for computing the steady-state distribution of a Markov chain. 
Let $P \in \Real^{N \times N}$ be the \emph{column}\footnote{A matrix is (row) stochastic if its entries are nonnegative and each row sums to one. A matrix is column stochastic if it is nonnegative and each column sums to one. Note that the usual convention in the probability literature is for the transition matrix to be \emph{row} stochastic, so if $X_t$ is a Markov chain with transition matrix $P$, then $\P[X_{t+1}=j \vert X_t =i ] = P_{ij}$. Following the literature on IAD, we adopt the opposite convention, taking $\P[X_{t+1}=j \vert X_t =i ] = P_{ji}$.} stochastic transition probability matrix of a discrete-time Markov chain on the state space
\begin{equation*}
  \Omega=\{1, \dots, N\}.
\end{equation*}
We call $\Omega$ the \emph{fine space}.
We assume that $P$ is irreducible, so there is a unique steady-state probability vector $\mu \in \Real^N$ solving
\begin{equation*}
  P \mu = \mu.
\end{equation*}

In each step of IAD, one calculates a coarse approximation to $P$ based on a user-specified partition of $\Omega$ into disjoint sets $\{S_i: i=1, \dots, n\}$. We call each set $S_i$ a \emph{coarse state}, and we call the set $\{1, \dots, n\}$ of indices the \emph{coarse space}. To define the coarse approximation, we specify operators mapping between the sets of probability vectors on the fine and coarse spaces. The \emph{aggregation operator} maps vectors on the fine space to vectors on the coarse space.

\begin{definition}
We define the \emph{aggregation operator} $A:\Real^N \rightarrow \Real^n$ by
\begin{equation*}
  (A \nu)_i := \nu^\t \1_{S_i} = \sum_{x \in S_i} \nu_x
\end{equation*}
 for any $i =1, \dots, n$  and  $\nu \in \Real^N$.
\end{definition}

Note that when $\nu$ is a probability vector, $A \nu_i$ is simply the probability of $S_i$ under $\nu$. Moreover, if $\nu$ is a probability vector, then so is $A\nu$.

The \emph{disaggregation operator} maps vectors on the coarse space to vectors on the fine space. It depends on the current approximation $\tilde \mu$ of the steady-state $\mu$. 

\begin{definition}
  Given a probability vector $\tilde \mu \in \Real^N$ with $A\tilde \mu >0$ and a coarse state $S_i$, define the conditional distribution $\tilde \mu(\cdot \vert S_i)$ by
  \begin{equation*}
    \tilde \mu (j \vert S_i) = \frac{\tilde \mu_j \1_{S_i}(j)}{A\tilde \mu_i}
  \end{equation*}
  for $j \in \Omega$. 
   Here, $\1_{S_i}$ denotes the characteristic function of $S_i$. 
  Define the \emph{disaggregation operator} $D(\tilde \mu):\Real^n \rightarrow \Real^N$ by
\begin{equation*}
  D(\tilde \mu) z_j :=  \sum_{i=1}^n z_i \tilde \mu (j \vert S_i)
\end{equation*}
for any $z \in \Real^n$.
\end{definition}

Note that if $z$ is a probability vector, so is $D(\tilde \mu)z$. Also, observe that $D(\tilde \mu)$ is defined only when $A\tilde \mu >0$. This will always be the case in practice under our assumptions, cf.~Lemma~\ref{lem: well-posedness}.  

Given an approximation $\tilde \mu$ of $\mu$, the coarse approximation $C(\tilde \mu)$ to $P$ is defined by composing $P$ with $A$ and $D(\tilde \mu)$.
\begin{definition}
Let $\tilde \mu \in \Real^N$ be a probability vector with $A \tilde \mu >0$. We define the \emph{coarse approximation} $C(\tilde \mu) \in \Real^{n \times n}$ by
\begin{equation*}
  C(\tilde \mu) =A P D(\tilde \mu).
\end{equation*}
\end{definition}
The coarse approximation $C(\tilde \mu)$ is a column stochastic matrix. To see this, note that $C(\tilde \mu)$ maps probability vectors to probability vectors, since each of $A$, $ P$, and $D(\tilde \mu)$  maps probability vectors to probability vectors. In each step of IAD, one solves for the steady-state of $C(\tilde \mu)$.  It is convenient to establish some general notation for this operation.

\begin{definition}
  For $M$ an irreducible and column stochastic matrix, we let $z(M)$ denote the unique probability vector solving
  \begin{equation*}
    z(M) = M z(M).
  \end{equation*}
\end{definition}
Now let $\mu^0 \in \Real^N$ be a user-specified initial approximation of $\mu$.
In IAD, one alternates coarse correction and smoothing steps.
Given a probability $\mu^k \in \Real^N$, the \emph{coarse correction} step is to compute
\begin{equation}\label{eq: coarse correction step}
  \mu^{k + \frac12} =  D(\mu^k) z(C(\mu^k)).
\end{equation}
To compute $z(C(\mu^k))$ in practice, one can use the algorithm outlined in Appendix~\ref{apx: computing coarse distribution} or other direct methods of numerical linear algebra.
The \emph{smoothing} step is to compute
\begin{equation}\label{eq: smoothing step}
  \mu^{k+1} = P \mu^{k + \frac12}.
\end{equation}
Note that the smoothing step is the same as one step of the power method for calculating $\mu$ and also the same as evolving $\mu^{k+\frac12}$ by one step under the forwards equation for the chain.

Many variations of IAD have appeared in the literature. We will not treat all of them.
Some versions apply to substochastic problems such as economic input-output models~\cite{mandel_local_1983,vakhutinsky_iterative_1979}. Others use different smoothers~\cite{krieger_two-level_1995,marek_local_1994,marek_convergence_1998}, apply multiple smoothing iterations at each step~\cite{krieger_two-level_1995,marek_local_1994,marek_convergence_1998}, smooth the aggregation or disaggregation operators~\cite{de_sterck_smoothed_2010}, use a hierarchy of several coarse approximations in a multigrid V or W-cycle~\cite{de_sterck_multilevel_2008}, or solve infinite-dimensional problems using more general aggregation and disaggregation operators~\cite{marek_local_1994,earle_convergence_2022}. We do not consider these possibilities.

We now summarize our assumptions and show that IAD is well-posed.
\begin{assumption}
  \label{asm: main assumptions}
  We assume:
  \begin{enumerate}
    \item $P$ is irreducible.  
  \item The initial approximation $\mu^0$ of the steady-state $\mu$ is strictly positive. 
  \end{enumerate}
\end{assumption}

\begin{lemma}
  \label{lem: well-posedness}
    If Assumption~\ref{asm: main assumptions} holds, then the iterates $\mu^k$ produced by IAD are defined for all $k \in \mathbb{N}$. In particular, $A \mu^k>0$ and $C(\mu^k)$ is irreducible, so $D(\mu^k)$ and $z(C(\mu^k))$ are defined.
\end{lemma}

\begin{proof}
See Appendix~\ref{apx: well-posedness}.
\end{proof}

If one does not assume $\mu^0 >0$, then $C(\mu^0)$ may be reducible, in which case the steady-state $z(C(\mu^0))$ need not be unique. This can occur even when $P$ is both irreducible and aperiodic; see Appendix~\ref{apx: pathological examples} for an example.

Finally, we present a complete version of IAD with a termination criterion similar to those typically used in practice. 

\begin{algorithm}
  The user must specify the following:
  \begin{enumerate}
  \item A column stochastic and irreducible transition matrix $P \in \Real^{N \times N}$.
  \item A partition $\{S_i: i =1, \dots, n\}$ of $\{1, \dots, N\}$ into disjoint sets.
  \item A probability vector $\mu^0 \in \Real^N$ with $\mu_0 >0$.
  \item An error tolerance $\tau>0$.
  \end{enumerate}
  Given these data, IAD proceeds as follows:
  \begin{enumerate}
  \item Set $ \mu^{\textrm{old}}=  \mu^0$.
  \item Calculate $z(C( \mu^{\textrm{old}}))$ using the algorithm in Appendix~\ref{apx: computing coarse distribution}. Set 
    \begin{equation*}
       \mu^{\textrm{new}} =  P D( \mu^{\textrm{old}}) z(C( \mu^{\textrm{old}})).
    \end{equation*}
  \item If 
      \begin{equation*}
      \max_{i \in 1, \dots, N} \frac{\lvert \mu^{\textrm{new}}_i - \mu^\textrm{old}_i \rvert}{\mu^\textrm{old}_i} \leq \tau \text{ and } \max_{i \in 1, \dots, N} \frac{\lvert P \mu^{\textrm{new}}_i - \mu^{\textrm{new}}_i \rvert}{P \mu^{\textrm{new}}_i} \leq \tau,
    \end{equation*}
    then output $ \mu^{\textrm{new}}$.
    Otherwise, set $ \mu^{\textrm{old}} =  \mu^{\textrm{new}}$, and return to step (2) above. 
  \end{enumerate}
\end{algorithm}

\section{Local Convergence of IAD}
\label{sec: local convergence}

In this section, we prove local convergence of IAD. That is, we show that if the initial approximation $\mu^0$ is sufficiently close to the true steady-state $\mu$, then $\mu^k$ converges to $\mu$. We begin with an analysis of the smoothing step of IAD in Section~\ref{subsec: power method}, which is the power method. In Section~\ref{subsec: local convergence}, we prove local convergence of IAD.  

\subsection{The Power Method}
\label{subsec: power method}

Here, we prove convergence of the power method (or equivalently convergence of the forwards equation of the chain) to the steady-state $\mu$. Of course, convergence of the power method is already well-understood. The details of our particular proof will be instrumental in our analysis of the efficiency of IAD, however. We begin by defining convenient norms.

\begin{definition}
  \label{def: weighted l2}
  Let $\nu \in \Real^k$ with $\nu >0$. We define the \emph{$\ell^2(\nu)$-norm and inner product} by 
  \begin{equation*}
    \langle x,y \rangle_{\nu} =  \sum_{i=1}^k x_i y_i \nu_i \text{ and } \lVert x \rVert_{\nu} = \langle x, x \rangle_{\nu}^{\frac12},
  \end{equation*}
  for $x, y \in \Real^k$.  Given $M\in \Real^{k \times k}$, we let $\lVert M \lVert_\nu$ denote the induced operator norm. We let $M^{\ast,\nu}$ be the adjoint matrix so that
  \begin{equation*}
    \langle M x, y \rangle_{\nu} = \langle x, M^{\ast,\nu} y \rangle_{\nu}
  \end{equation*}
  for all $x,y \in \Real^k$. 
\end{definition}

We will use the $\ell^2(1/\mu)$-norm to measure discrepancies between probability measures on $\Omega$, for example the error $\mu^k - \mu$ after $k$ steps of IAD. The $\ell^2(\mu)$-norm will arise when we analyze the efficiency of IAD. As a first step in our analysis of the power method, we relate adjoints in the $\ell^2(1/\mu)$-inner product to time reversals. 

\begin{lemma}
  \label{lem: time reversal}
  Let $P \in \Real^{N \times N}$ be column stochastic and irreducible.
  The time reversal of $P$ is $P^{\ast, 1/\mu}$. In particular,
  \begin{equation}
    \label{eq: formula for adjoint}
    P^{\ast, 1/\mu}  = \diag(\mu) P^\t \diag(1/\mu)
  \end{equation}
  is column stochastic, irreducible, and has invariant distribution $\mu$. 
\end{lemma}

\begin{proof}
  See Appendix~\ref{apx: time reversal}.
\end{proof}

The spectrum of the operator $P^{\ast,1/\mu} P$ will play a crucial role in our proof of convergence of the power method and in our efficiency analysis of IAD. By Lemma~\ref{lem: time reversal}, $P^{\ast, 1/\mu} P$ is column stochastic with invariant distribution $\mu$. Therefore, $\1$ is a left eigenvector of $P^{\ast ,1/\mu} P$ with eigenvalue $1$, and $\mu$ is the corresponding right eigenvector. (Here, $\1 \in \Real^{N}$ denotes the vector whose entries are all equal to one.) Moreover, $P^{\ast ,1/\mu} P$ is self-adjoint and positive semidefinite with respect to the $\ell^2(1/\mu)$-inner product. Therefore, $\sigma(P^{\ast ,1/\mu} P) \subset [0,1]$. Let
\begin{equation*}
  1 = \lambda_1 \geq \lambda_2 \geq \dots \geq \lambda_N \geq 0
\end{equation*}
be the eigenvalues of $P^{\ast ,1/\mu} P$ listed in decreasing order and with repetition if any have multiplicity greater than one.
Let $v_1, \dots, v_N$ be the corresponding right eigenvectors.
 Since $P^{\ast, 1/\mu} P$ is self-adjoint, we may assume that the eigenvectors are an orthonormal basis of $\Real^N$ with the $\ell^2(1/\mu)$-inner product, so
\begin{equation*}
   \langle v_i, v_j \rangle_{1/\mu} = \delta_{ij},
\end{equation*}
and we have the diagonalization 
\begin{equation}
  \label{eq: diagonalization}
 P^{\ast, 1/\mu} P = \mu \1^\t + \sum_{k=2}^N \lambda_k v_k v_k^\t \diag(1/\mu).
\end{equation}
We will refer to this diagonalization frequently. Note that the left eigenvectors of $P^{\ast, 1/\mu}P$ are
\begin{equation*}
  v_1'=\1, v_2'= \diag(1/\mu) v_2 , \dots, v_N' =\diag(1/\mu) v_N.
\end{equation*}
The left eigenvectors play an important role in our efficiency analysis of IAD.

We now show that when $P$ is irreducible, the power method for computing $\mu$ is strictly contracting in the $\ell^2(1/\mu)$-norm if and only if $P^\t P$ is irreducible.

\begin{lemma}
  \label{lem: convergence of power method}
  Let $P \in \Real^{N \times N}$ be an irreducible, column stochastic matrix. 
  Let $\nu^0 \in \Real^N$ be a probability vector, and define the power method iteration
\begin{equation*}
  \nu^{k+1} =  P \nu^k.
\end{equation*} 
Define
\begin{equation*}
  \hat P = P - \mu \1^\t.
\end{equation*}
We have
\begin{equation}
  \label{eqn: power method error}
  \nu^{k+1}- \mu = \hat P  (\nu^k - \mu).
\end{equation}
Moreover,
\begin{equation*}
  \lVert \hat P \rVert_{1/\mu} = \sqrt{\lambda_2},
\end{equation*}
and $\lambda_2 < 1$ if and only if $P^\t P$ is irreducible.
\end{lemma}

\begin{proof}
  See Appendix~\ref{apx: power method}.
\end{proof}

Note that asymptotic rate of convergence of the power method is
\begin{equation*}
  \lim_{m \rightarrow \infty} \lVert \hat P^m \rVert^{\frac1m} = \rho(\hat P).
\end{equation*}
For irreversible chains, $\rho(\hat P)$ may be significantly less than the contraction constant $\lVert \hat P \rVert_{1/\mu}$ of the power method in the $\ell^2(1/\mu)$-norm. See Section~\ref{subsec: iad one d irreversible} for an example. However, for reversible chains, we have $\rho(\hat P) =\lVert \hat P \rVert_{1/\mu}$, since $\hat P^{\ast, 1/\mu} = \hat P$.

In our convergence analysis of IAD, we assume that $P^\t P$ is irreducible. 
By formula~\eqref{eq: formula for adjoint} for $P^{\ast, 1/\mu}$, it is equivalent to assume that $P^{\ast, 1/\mu} P$ is irreducible. A sufficient, but not necessary, condition is that $P$ be irreducible with positive diagonal. We note that if $P$ is irreducible but $P^\t P$ is not, then $\bar P = \frac12(I+P)$ is irreducible, has a positive diagonal, and has the same unique steady-state as $P$. Therefore, to compute the steady-state of $P$ one could apply IAD with $\bar P$ in place of $P$, and local convergence would then be guaranteed by the results below.

We now give an example to illustrate what can go wrong when $P^\t P$ is reducible. Consider a right shift on three states:
\begin{equation*}
  P =
  \begin{pmatrix}
    0 & 0 &1 \\
    1 & 0 &0 \\
    0 &1 &0
  \end{pmatrix}.
\end{equation*}
This is an irreducible Markov chain, and the steady-state is the uniform distribution $\frac13 \1 \in \Real^3$. The time-reversal is the left shift $P^{\ast,\frac13\1} = P^\t = P^{-1}$.
Therefore, $P^{\ast,\frac13\1} P  = I$ is reducible even though both $P$ and $P^{\ast,\frac13\1}$ are irreducible. In this case, $\lVert \hat P \rVert_{1/\mu} = 1$, and the power method does not converge, since $P$ is periodic.

The right shift in our last example is irreducible, but periodic. There also exist irreducible, aperiodic chains so that $P^\t P$ is reducible. For such chains, the power method is convergent, but it is not a strict contraction in the $\ell^2(1/\mu)$-norm. See Appendix~\ref{apx: pathological examples} and~\cite{marek_note_2006} for an example of an irreducible, aperiodic chain so that $P^\t P$ is reducible and IAD is not locally convergent.

\subsection{Local Convergence of IAD}
\label{subsec: local convergence}
Here, we prove local convergence of IAD. We begin with a convenient reformulation of the steady-state problem for a Markov chain.

\begin{lemma}\label{lem: inverse formula for invariant distribution}
  Let $M \in \Real^{k\times k}$ be an irreducible column stochastic matrix, and let $v,w \in \Real^k$ with $\1^\t v \neq 0$ and $z(M)^t w\neq 0$. The matrix $I - M + v w^\t$ is invertible, and
  \begin{equation*}\label{eqn: inverse formula for invariant distribution}
    z(M) = (I - M + v w^\t)^{-1} v w^\t z(M).
  \end{equation*}
\end{lemma}

\begin{proof}
  See Appendix~\ref{apx: inverse formula}.
\end{proof}

We use two special cases of Lemma~\ref{lem: inverse formula for invariant distribution}. First, the steady-state $\mu \in \Real^N$ is the unique solution $x$ of
\begin{equation}\label{eq: linear formulation of fine problem}
  (I - P + \mu \1^\t) x = \mu.
\end{equation}
Second, the coarse steady-state $z(C(\mu^k)) \in \Real^n$ is the unique solution of
\begin{equation}\label{eq: linear formulation of coarse problem}
  (I - C(\mu_k) + A\mu \1^\t ) x = A\mu. 
\end{equation}
Note that one cannot solve the linear equations~\eqref{eq: linear formulation of fine problem} and~\eqref{eq: linear formulation of coarse problem} in practice to compute $\mu$ and $z(C(\mu^k))$, since both the matrices and the right-hand-sides depend on the unknown $\mu$. We use~\eqref{eq: linear formulation of fine problem} and~\eqref{eq: linear formulation of coarse problem} only to derive the following recursive formula for the error after the coarse correction. 

\begin{lemma}
  \label{lem: error propagation for coarse correction step}
  For any probability vector $\nu \in \Real^N$ with $\nu >0$, the matrix ${A (I - P + \mu \1^\t)  D(\nu)}$ is invertible, and we may define 
   \begin{equation*}
     S(\nu) :=  D(\nu)[A (I - P + \mu \1^\t)  D(\nu)]^{-1} A (I - P + \mu \1^\t).
   \end{equation*}
   The operator $S(\nu)$ is a projection with $\rg (S(\nu)) = \rg(D(\nu))$. We call $S(\nu)$ the \emph{coarse projection}.
   We have
   \begin{equation}
     \label{eq: error after coarse step}
     \mu^{k+\frac12} - \mu = (I - S(\mu^k)) (\mu^k - \mu).
   \end{equation}
\end{lemma}

\begin{proof}
See Appendix~\ref{apx: error propagation for coarse correction}.
\end{proof}

In Appendix~\ref{apx: iad as multigrid}, we interpret IAD as an adaptive algebraic multigrid method for solving~\eqref{eq: linear formulation of fine problem}. The coarse projection $S(\mu^k)$ is exactly the coarse grid correction in this interpretation. The restriction operator is $A$, and the prolongation operator is $D(\mu^k)$. Note that $S(\nu)$ is a projection on $\rg(D(\nu))$. This is the most important algebraic fact in our analysis below. All of our results below would hold if $S(\mu)$ were \emph{any} projection on $\rg(D(\mu))$, except for Theorem~\ref{thm: exact spectrum of J}. 

To derive a recursive formula for the error after a complete step of IAD, we simply compose formula~\eqref{eq: error after coarse step} for the error after the coarse correction with formula~\eqref{eqn: power method error} for the propagation of the error under the power method. 

\begin{lemma}
  \label{lem: error propagation formula}
 Define the \emph{error propagation operator}
  \begin{equation*}
    J(\mu^k) := \hat P (I - S(\mu^k)).
  \end{equation*}
  We have
  \begin{equation*}
    \mu^{k+1} - \mu = J(\mu^k) (\mu^{k} - \mu).
  \end{equation*}
\end{lemma}

\begin{proof}
  By~\eqref{eqn: power method error} and~\eqref{eq: error after coarse step}, we have
  \begin{align*}
    \mu^{k+1} - \mu &= \hat P ( \mu^{k+\frac12} - \mu) 
    = \hat P (I - S(\mu^k)) (\mu^k - \mu).
  \end{align*}
 Similar formulas for the propagation of the error appear in~\cite{mandel_local_1983} and subsequent work~\cite{marek_convergence_2003,marek_convergence_1998,marek_local_1994}.
\end{proof}

We now show that $J(\mu)$ has norm less than one with respect to a certain operator norm. Local convergence follows. To construct the right norm, we decompose $\Real^N$ into $\rg(D(\mu))$ and its orthogonal complement in $\ell^2(1/\mu)$. The orthogonal projection $\Pi(\mu)$ defined below will be useful.

\begin{definition}
  Given a probability vector $\nu \in \Real^N$ with $\nu >0$, we define the \emph{orthogonal coarse projection}
  \begin{equation*}
    \Pi(\nu) := D(\nu) A.
  \end{equation*}
\end{definition}

Lemma~\ref{lem: properties of multigrid projection} summarizes the properties of $\Pi(\nu)$.

\begin{lemma}
  \label{lem: properties of multigrid projection}
  Let $\nu \in \Real^N$ be a positive probability vector. The orthogonal coarse projection $\Pi(\nu)$ is the orthogonal projection on $\rg(D(\nu))$ with respect to $\langle, \rangle_{1/\nu}$. It is also a reversible, column stochastic matrix with invariant distribution $\nu$. 
\end{lemma}

\begin{proof}
For completeness, we give a proof in Appendix~\ref{apx: properties of multigrid projection}. Equivalent observations appear in~\cite{mandel_local_1983} and subsequent work~\cite{marek_convergence_2003,marek_convergence_1998,marek_local_1994}. 
\end{proof}

We will show that for the $\lVert \cdot \rVert_{\eps}$ norm defined below, $\lVert J(\mu) \rVert_\eps < 1$ for sufficiently small $\eps$.
\begin{definition}
  For $\eps >0$, we define the $\eps$-inner product and norm on $\Real^N$ by
  \begin{equation*}
    \langle x,y \rangle_\eps := \langle x,(I-\Pi(\mu)) y \rangle_{1/\mu} + \eps \langle x,\Pi(\mu) y \rangle_{1/\mu} \text{ and } \lVert x \rVert_\eps := \langle x,x \rangle_\eps^\frac12.
  \end{equation*}
For $M \in \Real^{N \times N}$, we let $\lVert M \rVert_\eps$ denote the induced operator norm.
\end{definition}

To verify that $\langle , \rangle_\eps$ is an inner product, observe that it is symmetric since $\Pi(\mu)$ is an orthogonal projection, hence $\Pi(\mu) = \Pi(\mu)^{\ast, 1/\mu}$.  It is nondegenerate since 
\begin{equation*}
  \begin{split}
    \lVert x \rVert_\eps^2 &\geq \min \{ 1, \eps \} \left \{ \lVert (I - \Pi(\mu) ) x \rVert_{1/\mu}^2 + \lVert \Pi (\mu) x \rVert_{1/\mu}^2 \right \}\\
    &= \min \{ 1, \eps \} \lVert x \rVert_{1/\mu}^2,
  \end{split}
\end{equation*}
using again that $\Pi(\mu)$ is an orthogonal projection. Bilinearity is inherited from $\langle, \rangle_{1/\mu}$.

We now show that when $\eps>0$ is small, $\lVert J(\mu) \rVert_\eps$ is approximately $\lVert (I- \Pi(\mu)) J(\mu) \rVert_{1/\mu}$.

\begin{lemma}\label{lem: eps norm of J}
  We have
  \begin{equation*}
    \lVert J(\mu) \rVert_\eps \leq \sqrt{\lVert (I - \Pi(\mu)) J(\mu) \rVert_{1/\mu}^2 + \eps \lVert \Pi(\mu) - S(\mu) \rVert_{1/\mu}^2}.
  \end{equation*}
\end{lemma}

\begin{proof}
 See Appendix~\ref{apx: proof of lemma on eps norm of J}.
\end{proof}

We will estimate $\lVert (I - \Pi(\mu)) J(\mu) \rVert_{1/\mu}$. By Lemma~\ref{lem: eps norm of J}, if $\lVert (I - \Pi(\mu)) J(\mu) \rVert_{1/\mu} <1$, then $\lVert J(\mu) \rVert_\eps <1$ for $\eps$ sufficiently small. 

\begin{theorem}\label{lem: first estimate of norm of projection of J onto conditionals}
  Assume that $P$ and $P^\t P$ are irreducible and that at least one coarse state contains more than one fine state.  We have  
  \begin{align}
    \lVert (I - \Pi(\mu)) J(\mu) \rVert_{1/\mu}^2
    &= 1
      - \inf_{z \in \rg (I-S(\mu))}
      \frac{\langle z , (I - \hat P^{\ast,1/\mu } \hat P) z \rangle_{1/\mu}}{\lVert (I - \Pi(\mu) )z \rVert_{1/\mu}^2}  \label{eq: for showing refinement can only help} \\
    &\leq 1 - \frac{1}{\lVert (I - \Pi(\mu)) (I - \hat P^{\ast,1/\mu } \hat P)^{-1}  (I - \Pi(\mu)) \rVert_{1/\mu}} \label{eq: upper bound on i minus pi times J} \\
    &< 1. \nonumber
  \end{align}
\end{theorem}

\begin{proof}
See Appendix~\ref{apx: proof of estimate of norm of projection of J}.
\end{proof}


We now prove local convergence.

\begin{theorem}
  \label{thm: local convergence}
  Assume that $P$ and $P^\t P$ are irreducible.
  For $\eps >0$ sufficiently small,
  \begin{equation}
    \label{eqn: bound on J in the eps norm}
    \lVert J(\mu) \rVert_\eps < 1.
  \end{equation}
  For any $\eps >0$ small enough that $ \lVert J(\mu) \rVert_\eps < 1$ and any $\eta \in (0, 1- \lVert J(\mu) \rVert_\eps]$, there exists $r >0$ so that if $\lVert \mu^0 - \mu \rVert_\eps \leq r$, then
  \begin{equation*}
    \lVert \mu^k - \mu \rVert_{\eps} \leq (\lVert J(\mu) \rVert_{\eps} + \eta)^k \lVert \mu^0 - \mu \rVert_{\eps}
  \end{equation*}  
  for all $k \in \mathbb{N}$.
\end{theorem}

\begin{proof}
  See Appendix~\ref{apx: local convergence}.
\end{proof}

\section{The Rate of Convergence}
\label{sec: convergence rate}

Here, we analyze the asymptotic rate of convergence of IAD. First, we show that the rate is bounded by the spectral radius $\rho(J(\mu))$. We then derive an upper bound on the spectral radius based on the results in Section~\ref{sec: local convergence}. Our upper bound is appealing and easy to interpret, but it significantly overestimates $\rho(J(\mu))$ for some irreversible processes. See Section~\ref{subsec: iad one d irreversible} for an example. Therefore, we also derive an exact formula for $\rho(J(\mu))$. Our exact formula could be the basis for a better understanding of the rate of convergence of IAD for irreversible processes, and it yields an interpretable exact expression for $\rho(J(\mu))$ when $P$ is reversible. 

We now show that $\rho(J(\mu))$ bounds the asymptotic rate of convergence. By the asymptotic rate of convergence, we mean the expression on the left hand side of~\eqref{eqn: first estimate of asymptotic convergence rate} below. 

\begin{lemma}
  \label{lem: asymptotic rate of convergence}
  Let $P$ and $P^\t P$ be irreducible, let $r>0$ be as in Theorem~\ref{thm: local convergence}, and assume that $\lVert \mu^0 - \mu \rVert_{\eps} < r$. For any norm $\lVert \cdot \rVert$ on $\Real^{N}$, we have
  \begin{equation}
    \label{eqn: first estimate of asymptotic convergence rate}
    \limsup_{n \rightarrow \infty} \lVert \mu^k - \mu \rVert^{1/k} \leq \rho(J(\mu)).
  \end{equation}
\end{lemma}

\begin{proof}
See Appendix~\ref{apx: asymptotic rate of convergence}.
\end{proof}

We now estimate $\rho(J(\mu))$. Our approach is based on comparing the orthogonal coarse projection $\Pi(\mu)$ with the orthogonal projection on the eigenvectors associated with the largest eigenvalues of $P^{\ast, 1/\mu} P$.

\begin{definition}
  Fix some $k < N$, and let $Q$ be the $\ell^2(1/\mu)$-orthogonal projection on the eigenvectors $v_1 , \dots, v_k$ associated with the $k$ largest eigenvalues of $P^{\ast, 1/\mu} P$. 
That is, define 
\begin{equation}
  \label{eq: def of Q}
  Q =  \mu \1^\t+ \sum_{i=2}^k v_i (v'_i)^\t.
\end{equation}
Here, $\{v_i: i=1, \dots, N\}$ and $\{v'_i: i = 1, \dots, N\}$ are the left and right eigenvectors of $P^{\ast, 1/\mu} P$ as in~\eqref{eq: diagonalization}.
\end{definition}

Note that $Q=Q^{\ast, 1/\mu}$ and $Q^2 = Q$, so $Q$ is an orthogonal projection in $\ell^2(1/\mu)$. Moreover, by~\eqref{eq: diagonalization}, we have $Q P^{\ast, 1/\mu}  P = P^{\ast, 1/\mu}  P Q$, and $\sigma(Q P^{\ast, 1/\mu}  P) = \{\lambda_1, \dots, \lambda_k\}$.

When thinking about $Q$, we suggest that the reader keep the family of Markov chains defined in Section~\ref{sec: simple markov chain model of odl} in mind. We devised this family of chains as a simple model of the reversible metastable diffusion processes encountered in molecular simulation. For these chains, $P^{\ast, 1/\mu} P$ typically has a small number $k$ of eigenvalues that are very close to one. The rest of the spectrum is much farther from one. That is,
\begin{equation*}
  \frac{1-\lambda_k}{1-\lambda_{k+1}} \ll 1.
\end{equation*}
 In our examples, we choose $Q$ to be the projection associated with these $k$ largest eigenvalues. 

Our estimate of $\rho(J(\mu))$ is expressed in terms of the angle from $\rg(Q^\t)$ to $\rg(\Pi(\mu)^\t)$ in the $\ell^2(\mu)$-inner product.

\begin{lemma}
  \label{def: sin theta}
  We have  $0 \leq \lVert (I - \Pi(\mu)^\t) Q^\t \rVert_{\mu} \leq 1$, and therefore we may define an angle  $\theta \in [0,\pi/2]$ by
  \begin{equation}
    \label{eqn: def of sin theta}
    \sin (\theta) := \lVert (I - \Pi(\mu)^\t) Q^\t \rVert_{\mu}.
  \end{equation}
Note that here the norm is weighted by $\mu$ not $1/\mu$.
\end{lemma}

\begin{proof}
  Both $\Pi^\t$ and $Q^\t$ are orthogonal projections with respect to the $\ell^2(\mu)$-inner product, since $\Pi$ and $Q$ are orthogonal projections with respect to $\ell^2(1/\mu)$. Therefore, $\lVert (I - \Pi(\mu)^\t) Q^\t \rVert_{\mu} \leq \lVert (I - \Pi(\mu)^\t) \rVert_{\mu} \lVert Q^\t \rVert_{\mu} = 1$, since any orthogonal projection must have norm equal to one.
\end{proof}

One can show that the angle defined above coincides with the typical definition
\begin{align*}
  \sin(\theta) &:= \text{gap}(\rg(Q^\t),\rg(\Pi^\t)) \\
               &= \max_{\substack{u \in \rg(Q^\t) \\ \lVert u \rVert_\mu = 1}} \min_{\substack{w \in \rg(\Pi^t) \\ \lVert w \rVert_\mu=1}} \lVert u-w \rVert_{\mu} \\
  &=  \max_{\substack{u \in \rg(Q^\t) \\ \lVert u \rVert_\mu = 1}} \lVert (I-\Pi^t)u \rVert_{\mu}.
\end{align*}
We do not prove this, since it will not be important below.

We understand $\theta$ as a measure of how well one can approximate elements of $\rg(Q^\t)$ within $\rg(\Pi(\mu)^\t)$. Note that
\begin{equation*}
\rg(\Pi^\t(\mu)) = \spn \{\1_{S_1}, \dots, \1_{S_n}\},
\end{equation*}
and that
\begin{equation*}
  \rg(Q^\t) = \spn \{ v'_1, \dots, v'_k\} .
\end{equation*}
That is, $\rg(\Pi(\mu)^\t)$ is spanned by the characteristic functions of the coarse states, and $\rg (Q^\t)$ is spanned by the first $k$ left eigenvectors of $P^{\ast,1/\mu} P$. 
Therefore, $\theta$ will be small when each of the first $k$ left eigenvectors of $P^{\ast,1/\mu} P$ is well-approximated by a linear combination of characteristic functions of coarse states. Equivalently, $\theta$ is small when each of the first $k$ left eigenvectors can be approximated by a function that is constant on the coarse states.

We now estimate the asymptotic rate of convergence.

\begin{theorem}
  \label{thm: estimate of rate of convergence with spectrum and angle involved}
  Assume that $P$ and $P^\t P$ are irreducible and that at least one coarse state contains more than one fine state. We have
  \begin{align}    
    \rho(J(\mu))^2 &\leq \lVert (I - \Pi(\mu)) J(\mu) \rVert_{1/\mu}^2 \nonumber \\
    &\leq 1 - \frac{1}{\lVert (I - \Pi(\mu)) (I - \hat P^{\ast,1/\mu } \hat P)^{-1}  (I - \Pi(\mu)) \rVert_{1/\mu}} \label{eq: norm upper bound on rho} \\
    &\leq 1 - \frac{1}{\sin^2(\theta) \frac{1}{1- \lambda_2} + \cos^2(\theta) \frac{1}{1- \lambda_{k+1}}}. \label{eq: angle upper bound on rho}
  \end{align}
\end{theorem}

\begin{proof}
  See Appendix~\ref{apx: estimate of rate with spectrum and angle}. 
\end{proof}


In our examples in Section~\ref{sec: examples}, we consider metastable chains for which $P^{\ast, 1/\mu}P$ has a small number of eigenvalues very close to one, and we choose $k$ to be the number of such eigenvalues. We note however that Theorem~\ref{thm: estimate of rate of convergence with spectrum and angle involved} holds for any $k$. To obtain a useful estimate, one has to choose $k$ carefully. If $k$ is too large, then $\sin(\theta)$ will be close to one, giving only $\sqrt{\lambda_2}$ as an upper bound. 

Note that the right-hand-side of~\eqref{eq: angle upper bound on rho} increases from $\sqrt{\lambda_{k+1}}$ to $\sqrt{\lambda_2}$ as $\sin^2(\theta)$ increases from zero to one. Thus, the rate of convergence is never larger than $\sqrt{\lambda_2}$, which is the contraction constant of the power method in the $\ell^2(1/\mu)$-norm.
For reversible chains, $\sqrt{\lambda_2}$ is also the asymptotic rate of convergence of the power method, since any reversible $P$ is self-adjoint with respect to the  $\ell^2(1/\mu)$-inner product by Lemma~\ref{lem: time reversal} and therefore $\lVert \hat P \rVert_{1/\mu} = \rho(\hat P)$. Thus, for reversible chains, the asymptotic rate of convergence of IAD is never greater than the asymptotic rate for the power method. 

For irreversible chains, however, the asymptotic rate of convergence $\rho(\hat P)$ of the power method may be less than the contraction constant $\lVert \hat P \rVert_{1/\mu}$. Theorem~\ref{thm: estimate of rate of convergence with spectrum and angle involved} may significantly overestimate $\rho(J(\mu))$ in such cases. For example, suppose there is only a single coarse state $S_1= \Omega$. In that case, IAD reduces to the power method, and $J(\mu) = \hat P$. Moreover, $I- \Pi(\mu) = I - \mu \1^\t$, so  $(I - \Pi(\mu)) J(\mu) = \hat P$. Note that our upper bounds in Theorem~\ref{thm: estimate of rate of convergence with spectrum and angle involved} are in fact upper bounds on $\lVert (I - \Pi(\mu)) J(\mu) \rVert_{1/\mu}$. Therefore, if there is only one coarse state, neither of our upper bounds can be smaller than $\lVert \hat P \rVert_{1/\mu}$. See Figures~\ref{fig: slightly irreversible shift study} and~\ref{fig: moderately irreversible shift study} for additional examples where our upper bounds overestimate $\rho(J(\mu))$. 

Since our upper bound in Theorem~\ref{thm: estimate of rate of convergence with spectrum and angle involved} may significantly overestimate the spectral radius for some irreversible chains, we also give an exact formula for the spectrum of $J(\mu)$. This formula could lead to a better understanding of IAD in the irreversible case, and it also leads to an interpretable exact formula for $\rho(J(\mu))$ for reversible processes. 

\begin{theorem}
\label{thm: exact spectrum of J}
Assume that $P$ and $P^\t P$ are irreducible. Assume that there is more than one coarse state and at least one coarse state contains more than one fine state. The spectrum of $J(\mu)$ is given by 
\begin{equation}
\sigma(J(\mu))=\left(1-\frac{1}{\sigma((I-\Pi(\mu))(I-\hat{P})^{-1}(I-\Pi(\mu)))\setminus\{0\}}\right)\cup\{0\}\label{eq: exact spectrum formula}
\end{equation}
\end{theorem}

\begin{proof}
See Appendix~\ref{apx: proof of exact spectrum}.
\end{proof}

For reversible processes, Theorem~\ref{thm: exact spectrum of J} has the following corollary.

\begin{corollary}
  \label{cor: estimate of rate in reversible case}
  Let $P$ be reversible, and assume that $P$ and $P^\t P$ are irreducible. Assume that there is more than one coarse state and at least one coarse state contains more than one fine state.
  We have
  \begin{align}
    \rho(J(\mu)) &= 1 - \frac{1}{\lVert (I- \Pi(\mu)) (I - \hat P)^{-1} (I - \Pi(\mu)) \rVert_{1/\mu}}
    \label{eqn: exact norm formula for rho in reversible case}\\ 
                 & \leq  1 - \frac{1}{\sin^2(\theta) \frac{1}{1- \sqrt{\lambda_2}} + \cos^2(\theta) \frac{1}{1- \sqrt{\lambda_{k+1}}}}
    \label{eqn: upper bound in terms of theta in reversible case corollary}.
  \end{align}
\end{corollary}

\begin{proof}
See Appendix~\ref{apx: corollary for reversible case}.
\end{proof}

We include the angle upper bound~\eqref{eqn: upper bound in terms of theta in reversible case corollary} in Corollary~\ref{cor: estimate of rate in reversible case} to demonstrate that the exact formula~\eqref{eqn: exact norm formula for rho in reversible case} can be interpreted in the same way as the norm upper bound~\eqref{eq: norm upper bound on rho} in Theorem~\ref{thm: estimate of rate of convergence with spectrum and angle involved}. There does not appear to be a meaningful difference between the two angle upper bounds in Theorem~\ref{thm: estimate of rate of convergence with spectrum and angle involved} and Corollary~\ref{cor: estimate of rate in reversible case} when $\lambda_{2}$ and $\lambda_{k+1}$ are both close to one. In fact, one can show that the two angle upper bounds are asymptotic in various limits as $\lambda_2$ and $\lambda_{k+1}$ tend to one.

We now list three implications of our theory for the choice of coarse states. First, Corollary~\ref{cor: estimate of rate in reversible case} suggests that \emph{for a reversible process one should choose coarse states so that the left eigenvectors of $P$ corresponding to eigenvalues close to one are well-approximated by vectors that are constant on the coarse states}. In Section~\ref{sec: examples}, we explain how to interpret this statement for processes like those used in molecular modeling. Similarly, Theorem~\ref{thm: estimate of rate of convergence with spectrum and angle involved} suggests that for irreversible processes one should choose coarse states so that the leading left eigenvectors of $P^{\ast,1/\mu} P$ are well-approximated. Our examples in Section~\ref{subsec: iad one d irreversible} indicate that this may be good advice for some irreversible chains, but that it could be misleading for some very irreversible chains, cf.\@ Figure~\ref{fig: moderately irreversible shift study}.

Second, since the quality of approximation is measured in the $\ell^2(\mu)$-norm, \emph{one only needs an accurate approximation in regions of high probability under $\mu$.} Regions of low probability will not have a significant influence on $\sin(\theta)$ unless some of the first $k$ eigenvectors are concentrated in those regions. As a consequence, the efficiency of IAD is not always as sensitive to the choice of coarse states as one might expect, and in some cases a very na\"ive choice of coarse states can work quite well. See Section~\ref{subsec: 2d computations} for an example. 

Third, Corollary~\ref{cor: estimate of rate in reversible case} proves that \emph{for reversible chains $\rho(J(\mu))$ decreases whenever the coarse states are refined.} We say that a set of coarse states $\mathcal{R} = \{T_1, \dots, T_m\}$ is a refinement of $\mathcal{C} = \{S_1, \dots, S_n\}$ if each $S_i$ can be expresssed as a union of $T_j$'s. Let $\Pi_\mathcal{R}$ and $\Pi_\mathcal{C}$ be the orthogonal coarse projections $\Pi(\mu)$ for the two partitions $\mathcal{R}$ and $\mathcal{C}$, respectively. To see that the spectral radius $\rho(J(\mu))$ for the refined partition $\mathcal{R}$ is less than or equal to the spectral radius for the coarse partition $\mathcal{C}$, observe that $\rg(\Pi_\mathcal{R}) \supset \rg(\Pi_\mathcal{C})$, so
\begin{equation*}
  \Pi_\mathcal{C} \Pi_\mathcal{R} = \Pi_\mathcal{R} \Pi_\mathcal{C} = \Pi_\mathcal{C},
\end{equation*}
since both $\Pi_\mathcal{C}$ and $\Pi_\mathcal{R}$ are  $\ell^2(1/\mu)$-orthogonal projections. Therefore,
\begin{align*}
  &\lVert (I-\Pi_\mathcal{R}) (I-\hat P)^{-1} (I - \Pi_\mathcal{R}) \rVert_{\frac{1}{\mu}} \\
  \qquad &=  \lVert (I-\Pi_\mathcal{R}) (I-\Pi_\mathcal{C}) (I-\hat P)^{-1} (I-\Pi_\mathcal{C}) (I - \Pi_\mathcal{R}) \rVert_{\frac{1}{\mu}} \\
  \qquad &\leq \lVert I-\Pi_\mathcal{R} \rVert_{\frac{1}{\mu}} \lVert (I-\Pi_\mathcal{C}) (I-\hat P)^{-1} (I-\Pi_\mathcal{C})  \rVert_{\frac{1}{\mu}} \lVert I-\Pi_\mathcal{R} \rVert_{\frac{1}{\mu}} \\
  \qquad &\leq  \lVert (I-\Pi_\mathcal{C}) (I-\hat P)^{-1} (I-\Pi_\mathcal{C})  \rVert_{\frac{1}{\mu}},
\end{align*}
since $ I-\Pi_\mathcal{R}$ is an $\ell^2(1/\mu)$-orthogonal projection and so $ \lVert I-\Pi_\mathcal{R} \rVert_{\frac{1}{\mu}}=1$. It follows by Corollary~\ref{cor: estimate of rate in reversible case} that the spectral radius for the refined partition is less than for the coarse partition.

For irreversible chains, the spectral radius may increase with refinement. See Section~\ref{subsec: iad one d irreversible} for an example. However, in our examples, we still observe a clear (but not monotone) trend toward lower spectral radii with increasing refinement. We also note that all of the upper bounds on the spectral radius in Theorem~\ref{thm: estimate of rate of convergence with spectrum and angle involved} must decrease with refinement even when $P$ is irreversible.

\section{Examples Related to Modeling Molecules}
\label{sec: examples}

Here, we apply the theory developed in Section~\ref{sec: convergence rate} to develop an understanding of the rate of convergence of IAD for processes similar to those used as molecular models. To begin, we review certain important properties of molecular models, and we define a simple family of Markov chains with similar properties. We then calculate $\rho(J(\mu))$ and the upper bounds in Theorem~\ref{thm: estimate of rate of convergence with spectrum and angle involved} for some members of this family and for various choices of coarse states. Our theory explains the observed dependence of the rate of convergence on the choice of coarse states for all but the most irreversible (and least metastable) chains.

\subsection{Molecular Models}
\label{subsec: molecular models}

Molecular modeling begins with the specification of a potential energy $V : \Real^{M} \rightarrow \Real$ defined on the space of all configurations of the the atoms comprising the system. Based on the potential, one defines a stochastic process to model the evolution of the system. For example, the overdamped Langevin dynamics
\begin{equation*}
  dX_t = - \nabla V(X_t) \, dt + \sqrt{ 2 k T } \, d B_t
\end{equation*}
may be used to model a system in contact with a heat bath at temperature $T$. (Here, $k$ is Boltzmann's constant.) Refer to~\cite{lelievre_partial_2016} for details. We recall the following well-known properties of overdamped Langevin:
\begin{itemize}
\item Under some conditions on $V$, the unique steady-state of $X_t$ is the \emph{Boltzmann distribution}
  \begin{equation*}
    \pi(dx) = Z^{-1} \exp \left ( \frac{V(x)}{kT} \right ) \, dx \text{ where } Z^{-1} = \int_{\Real^{3N}} \exp(-\beta V(x)) \, dx.
  \end{equation*}
\item $X_t$ is \emph{reversible}.
\item If the potential energy $V$ has several local minima, then when the temperature $T$ is low, $X_t$ is \emph{metastable}. In particular, trajectories tend to vibrate around local minima of $V$, undergoing transitions between minima only rarely. Under some conditions on $V$, in the limit as $T \rightarrow 0$, each local minimum of $V$ corresponds to an eigenvalue of the generator of $X_t$ that converges exponentially to zero. The remainder of the spectrum remains bounded away from zero uniformly in $T$. The eigenvectors corresponding to the eigenvalues that converge to zero are approximately constant on the basins of attraction of the minima. See~\cite[Section~2.5]{lelievre_partial_2016} for details. 
\end{itemize}

Overdamped Langevin is reversible, but we take a particular interest in irreversible models, since these are the hardest to sample. For example, consider 
\begin{equation}
  \label{eq: odl with nonconservative force}
  dX_t = (-\nabla V(X_t) +  \alpha F(X_t) ) \,  dt + \sqrt{2 k T} \, dB_t,
\end{equation}
where $F$ is a nonconservative force, i.e.\@ $F$ is not the gradient of a potential function. Here, $X_t$ is irreversible~\cite[Section~5.1.2]{lelievre_partial_2016}. There is no general, closed-form expression for the steady-state density of~\eqref{eq: odl with nonconservative force}. In particular, the steady-state is not the Boltzmann distribution. This is one reason why sampling nonequilibrium steady-states is difficult.

\subsection{Simple Markov Chain Model of Overdamped Langevin}
\label{sec: simple markov chain model of odl}

We define a family of Markov chains on a one-dimensional grid with properties similar to overdamped Langevin. Let $V : \Real \rightarrow \Real$, $T >0$, $[a,b ] \subset \Real$, and $N \in \mathbb{N}$. Define the discrete Boltzmann distribution $\mu \in \Real^N$ by
\begin{equation}
  \label{eq: discrete Boltzmann}
  \mu_i = Z^{-1} \exp \left (- \frac{ V \left ( a + \frac{b-a}{N} i \right )}{T} \right ),
\end{equation}
where  $Z =\sum_{i =1}^N \exp \left (- \frac{ V \left (  a + \frac{b-a}{N}i \right )}{T} \right )$.
Define the transition matrix 
\begin{alignat}{3}
P_{ii} &:= \frac{1}{2} \left (\frac{\mu(i)}{\mu(i-1) + \mu(i)} + \frac{\mu(i)}{\mu(i+1) + \mu(i)} \right )
&&\mbox{ for all } i\in \Omega,  \nonumber \\
P_{i+1,i} &:= \frac{1}{2}\frac{\mu(i+1)}{\mu(i+1) + \mu(i)} &&\mbox{ for all } i\in \Omega, \label{eq: example reversible transition matrix}\\
P_{i-1,i} &:= \frac{1}{2}\frac{\mu(i-1)}{\mu(i-1) + \mu(i)} &&\mbox{ for all } i\in \Omega, \mbox{ and } \nonumber \\
P_{ji} &:= 0 &&\mbox{ otherwise.} \nonumber
\end{alignat}
In the definition of $P$, we impose periodic boundary conditions, associating $0$ with $N$, $1$ with $N+1$, etc. This family of Markov chains was proposed in~\cite{thiede_sharp_2015} as a model of overdamped Langevin and other metastable processes often encountered in statistical physics. We also define a similar family of chains on a two-dimensional grid; see Appendix~\ref{apx: 2d chain definition} and Section~\ref{subsec: 2d computations}.

The Markov chain $P$ has properties similar to overdamped Langevin:
Observe that $P$ is in detailed balance with the discrete Boltzmann distribution $\mu$, so $P$ is reversible and has invariant distribution $\mu$. In our examples below, we choose $T$ small, and in that case $P$ is metastable, as demonstrated in~\cite{thiede_sharp_2015}. Moreover, in the examples given in Section~\ref{subsec: 1d computations}, for each local minimum of $V$, there is one eigenvalue of $P$ that lies very close to one and the remainder of the spectrum lies much farther from one. The left eigenvectors of $P$ associated with the slow eigenvalues are approximately constant on the basins of attraction of the minima. We will not prove that these properties of the spectrum hold in general (or even formulate them precisely), but we note that they do hold in our examples.

We also define irreversible Markov chains that are analogous to overdamped Langevin with a nonconservative force~\eqref{eq: odl with nonconservative force}. Define the \emph{right shift} $W \in \Real^{N \times N}$ by  
\begin{equation}
  \label{eq: right shift matrix}
  \begin{split}
    W_{i+1,i} &:= 1 \text{ for all } i \in \Omega, \text{ and } \\
    W_{j,i} &:= 0 \text{ otherwise,}
  \end{split}
\end{equation}
taking periodic boundary conditions as in the definition of $P$.
We consider chains of the form $(1 - \alpha) P + \alpha W$ for $\alpha \in (0,1)$.

\subsection{IAD for a Metastable, Reversible Chain on a One-Dimensional Grid}
\label{subsec: 1d computations}

We now test our theory on a highly metastable, reversible problem.
We will see that for any sufficiently refined choice of coarse states, IAD converges quickly compared with the power method. However, for some very poor choices of coarse states, IAD converges at essentially the same rate as the power method. We explain these results in detail using the rate estimate in terms of $\theta$~\eqref{eq: angle upper bound on rho} and the properties of molecular models outlined above.

Let $\mu$ be the discrete Boltzmann distribution defined in~\eqref{eq: discrete Boltzmann} with 
\begin{equation*}
V(x) =(1-x^{2})^{2}+\frac{1}{2}x,
\end{equation*}
$N = 100$, $[a,b]= [-1.7,1.55]$, and $T=1/10$. Let $P$ be the corresponding reversible transition matrix defined by~\eqref{eq: example reversible transition matrix}. Here, the potential $V$ has two minima, so we expect that exactly two eigenvalues of $P^{\ast, 1/\mu} P= P^2$ will lie very close to one with the remainder significantly farther from one. Table~\ref{tab: reversible, asymmetric 1-d eigenvalues} confirms that this is indeed the case. The left eigenvectors are displayed in Figure~\ref{fig: reversible 1D slow modes}. Based on our discussion of molecular models above, we expect that the eigenvectors corresponding to the two largest eigenvalues should be approximately constant on the basins of attraction of $V$. Here, on the grid used to define $\mu$, $V$ has a local maximum at $i=57$. It has a global maximum at $i=0$, which is identified with $i=100$ by periodicity. These maxima divide the state space into two basins of attraction. Observe that the first two eigenvectors, $v_1' = \1$ and $v_2'$, are roughly constant on the basins of attraction. The third is not. 

\begin{table}[h]
  \label{tab: reversible, asymmetric 1-d eigenvalues}
    \caption{The largest five eigenvalues of $P^{\ast, 1/\mu} P$ for the reversible, one-dimensional chain $P$ of Section~\ref{subsec: 1d computations}. We report $\sqrt{\lambda_k}$ instead of $\lambda_k$, since it is $\sqrt{\lambda_k}$ that appears in Theorem~\ref{thm: estimate of rate of convergence with spectrum and angle involved}. We have added the third column for comparison with Figure~\ref{fig: reversible 1D bound comparison}.}
  \begin{tabular}{lll}
    \toprule
    $k$ & $\sqrt{\lambda_k}$ & $-\log_{10}(1-\sqrt{\lambda_k})$ \\
    \midrule
    $1$ & $1$ & $\infty$ \\
    $2$ & $0.999992$ & $5.09$ \\
    $3$ & $0.991441$ & $2.07$ \\
    $4$ & $0.986243$ & $1.86$ \\
    $5$ & $0.979807$ & $1.69$ \\
          \bottomrule
  \end{tabular}
\end{table}

\begin{figure}[h]
  \label{fig: reversible 1D slow modes}
  \caption{Left: the steady-state $\mu$. Right: the eigenvectors of $P^{\ast, 1/\mu} P$. Note that $v'_1$ and $v'_2$ are approximately constant on the basins of attraction of the local minima of $V$, which correspond to maxima of the steady-state distribution. Here, the local minima of $\mu$ at $i=57$ and $i=0$ separate the basins of attraction. Note that $i=0$ and $i=100$ are identified, since we impose periodic boundary conditions.}
  \begin{center}
    \includegraphics{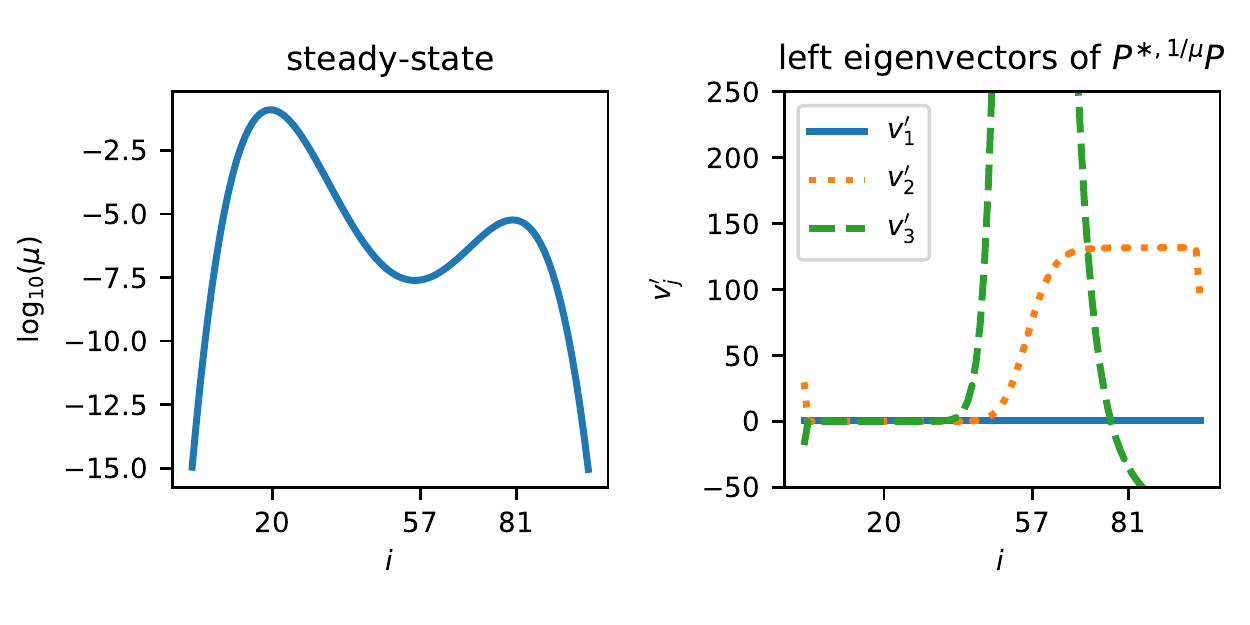}
  \end{center}
\end{figure}

We compute $\rho(J(\mu))$ and the upper bounds in Theorem~\ref{thm: local convergence} for several different choices of coarse states. First, we test uniform grids. We show for this simple, one-dimensional system that IAD converges quickly whenever the coarse states are sufficiently refined. Therefore, one does not need detailed prior knowledge of the eigenvectors of $P$ to choose good coarse states for this problem. For any $n \in \mathbb{N}$ and $\ell \in \{0, \dots, \lfloor 100/n \rfloor\}$, we define the uniform grid of coarse states 
\begin{equation}
  \label{eq: uniform grid}
    S_J = \left \{\left \lfloor J \frac{100}{n} \right \rfloor+\ell, \dots, + \left \lfloor (J+1) \frac{100}{n} \right \rfloor + \ell-1 \right \}
\end{equation}
for $J = 0, \dots, n-2$, and 
\begin{equation*}
  S_n = \{0, \dots, \ell-1\} \cup \left \{\left \lfloor (n-1) \frac{100}{n} \right \rfloor, \dots, 99 \right \}.
\end{equation*}
For example, for $n=2$ and $\ell=5$, the coarse states would be $S_0=\{5,\dots, 54\}$ and $S_1=\{0,\dots, 4\} \cup \{55, \dots, 99\}$. In Figure~\ref{fig: refinement study}, we report the maximum of $\rho(J(\mu))$ over all $l=0, \dots, \lfloor 100/n\rfloor$ for $n=1, \dots, 20$. We see a clear decreasing trend with a growing number of coarse states.    

\begin{figure}[h]
  \label{fig: refinement study}
  \caption{The maximum of $\rho(J(\mu))$ over all $l=0, \dots, \lfloor 100/n\rfloor$ for $n=1, \dots, 20$ for the uniform grids of coarse states defined in~\eqref{eq: uniform grid}. The curve labeled $\alpha=0$ is for the reversible chain of Section~\ref{subsec: 1d computations}. The other curves are for the irreversible chains in Section~\ref{subsec: iad one d irreversible}.}
  \begin{center}
    \includegraphics{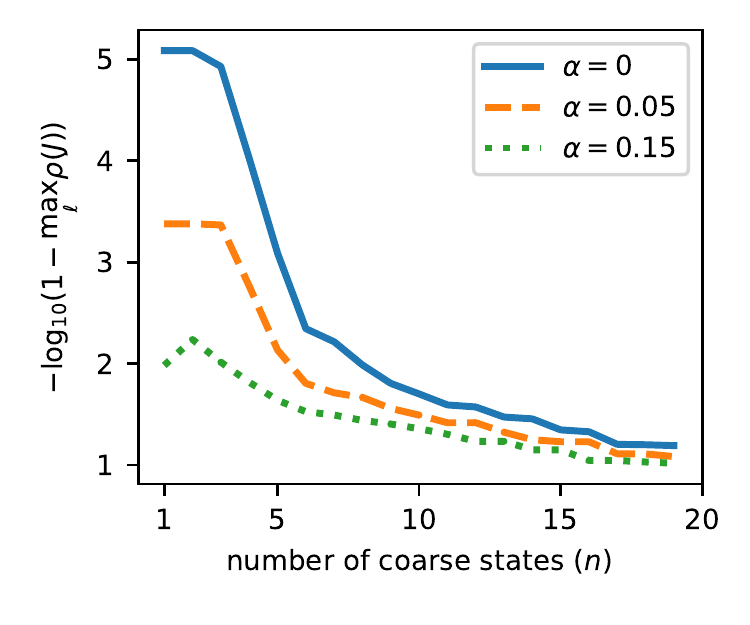}
  \end{center}
\end{figure}



We now investigate the dependence of $\rho(J(\mu))$ on the choice of coarse states in more detail. We compute the spectral radius $\rho(J(\mu))$ and our upper bounds for a family of coarse states of the form
\begin{equation}
\label{eq: shifted coarse states}
S_{1}=\{0,...,\ell\} \text{ and } S_{2}=\{\ell+1,...,99\},
\end{equation}
with $\ell=0, \dots, 98$. We display the results in Figure~\ref{fig: reversible 1D bound comparison}. Different locations $\ell$ of the boundary between coarse states result in different angles $\theta$, depending on how well $v'_2$ can be approximated by vectors that are constant on the coarse states. We expect $\theta$ to be small when the boundary between the coarse states coincides with the boundary between the basins of attraction, and this happens when $\ell=57$. Note that when $\ell$ is close to $57$, $\rho(J(\mu)) \approx \sqrt{\lambda_3}$, and it is as if one has eliminated the larger eigenvalue $\sqrt{\lambda_2}$. When $\ell$ is far from $57$, $\rho(J(\mu)) \approx \sqrt{\lambda_2}$, and IAD will converge at approximately the same rate as the power method. Note that both of the upper bounds in Theorem~\ref{thm: estimate of rate of convergence with spectrum and angle involved} yield precise estimates of $\rho(J(\mu))$, but the  norm bound~\eqref{eq: norm upper bound on rho} is so precise as to be indistinguishable from the spectral radius $\rho(J(\mu))$.

Note that the optimal coarse states in the family~\eqref{eq: shifted coarse states} considered above coincide with the basins of attraction of $V$. We wish to emphasize that it is not in general necessary to choose the coarse states to be the basins of attraction. It is only necessary that the leading left eigenvectors of $P$ be well-approximated by functions that are constant on the coarse states. Note that this will be true whenever the coarse states are sufficiently refined.

\begin{figure}[h]
  \caption{The spectral radius $\rho(J(\mu))$, the norm upper bound~\eqref{eq: norm upper bound on rho}, and the angle upper bound~\eqref{eq: angle upper bound on rho} for the reversible, one-dimensional chain. Each of these numbers $x$ is very close to one for all values of $\ell$, so we plot $-\log_{10}(1-x)$. The variable $\ell$ on the horizontal axis relates to the definition of the coarse states, cf.\@ equation~\eqref{eq: shifted coarse states}. The spectral radius and the norm bound are indistinguishable in this figure.}
  \label{fig: reversible 1D bound comparison}
  \begin{center}
    \includegraphics{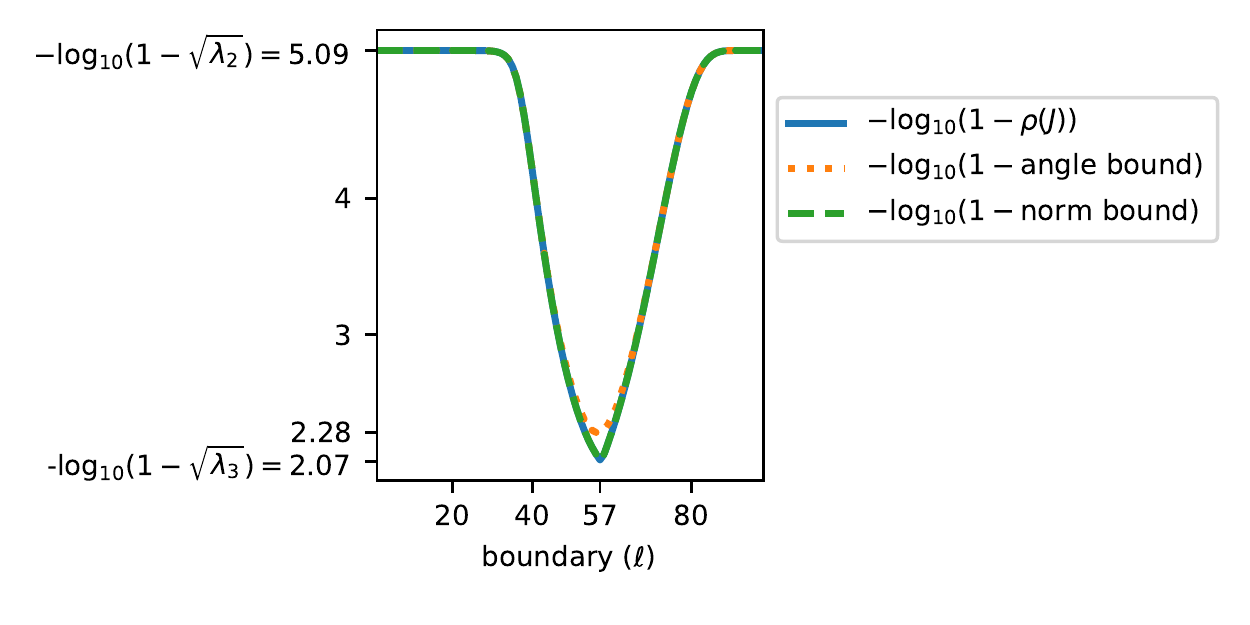}
  \end{center}
\end{figure}


\subsection{IAD for Irreversible Chains on a One-Dimensional Grid}
\label{subsec: iad one d irreversible}
We now consider irreversible perturbations of the last process.
Define 
\begin{equation*}
P_\alpha =(1- \alpha)   P  + \alpha  W,
\end{equation*}
where $P$ is as above, $W$ is the right shift matrix~\eqref{eq: right shift matrix}, and $\alpha \in [0,1]$. Let $\mu_\alpha$ be the steady-state of $P_\alpha$. We compute $\rho(J(\mu_\alpha))$ and the upper bounds for the families of coarse states defined above in Section~\ref{subsec: 1d computations}. We will see that our bounds are not as precise for irreversible chains as reversible chains. For $\alpha=0.05$, our bounds overestimate the true rate of convergence, but correctly predict the dependence of the rate on the choice of coarse states. For $\alpha=0.15$, our bounds do not seem to yield any useful information about the dependence of the rate of convergence on the coarse states. However, we note that $P_{0.15}$ is not metastable: we have $\rho(\hat P_{0.15})\approx 0.99$ compared with $\rho(\hat P) \approx 0.99999$, cf.\@ Table~\ref{tab: rates of convergence as function of alpha}. Therefore, one does not need a sophisticated method like IAD to estimate the steady-state of a kernel like $P_{0.15}$, so $P_{0.15}$ is not of much interest as a test case for IAD. We include results for $\alpha=0.15$ simply to illustrate the limitations of our theory.

In Figure~\ref{fig: refinement study}, we report the maximum of $\rho(J(\mu_\alpha))$ over all $l=0, \dots, \lfloor 100/n\rfloor$ for $n=1, \dots, 20$. We see a clear decreasing trend with a growing number of coarse states. However, note that the spectral radius does not decrease monotonically with the number of coarse states for $\alpha = 0.15$. In particular, for some choices of coarse states with $n=2$, $\rho(J(\mu_{0.15}))$ is larger than $\rho(\hat P_{0.15})$. For such poor choices of coarse states, IAD would converge more slowly than the power method.  

In Figure~\ref{fig: slightly irreversible shift study}, we report $\rho(J(\mu_\alpha))$ and the upper bounds for the family of shifted coarse states~\eqref{eq: shifted coarse states} for $\alpha=0.05$. We report the steady-state $\mu_\alpha$ and the three leading left eigenvectors of $P_\alpha$ in Figure~\ref{fig: slightly irreversible shift study}. Note the similarity with the eigenvectors of $P_0$ in Figure~\ref{fig: reversible 1D slow modes}. 
Our upper bounds correctly predict the dependence of $\rho(J(\mu_\alpha))$ on $\ell$. However, note that when $\ell$ is far from the optimal value, $\rho(J(\mu_\alpha))$ is almost the same as $\rho(\hat P_\alpha)$, which is the asymptotic rate of convergence of the power method. Our estimates in Theorem~\ref{thm: estimate of rate of convergence with spectrum and angle involved} predict the slower rate of convergence $\sqrt{\lambda_2}$, which is the contraction constant of the power method in $\ell^2(1/\mu)$. Recall that all of our upper bounds on $\rho(J(\mu))$ are in fact upper bounds on $\lVert (I - \Pi(\mu)) J(\mu) \rVert_{1/\mu}$. Note that although our bounds are generally quite close to $\lVert (I - \Pi(\mu)) J(\mu) \rVert_{1/\mu}$, they sometimes significantly overestimate $\rho(J(\mu))$, cf.\@ the discussion after the statement of Theorem~\ref{thm: estimate of rate of convergence with spectrum and angle involved}.  

In Figure~\ref{fig: moderately irreversible shift study}, we report $\rho(J(\mu_\alpha))$ and the upper bounds for the family of shifted coarse states~\eqref{eq: shifted coarse states} for $\alpha=0.15$. 
Here, our upper bounds do not seem to yield any useful information about the dependence of the convergence rate on the choice of coarse states. We propose that more precise estimates based on the exact formula for the spectral radius given in Theorem~\ref{thm: exact spectrum of J} could be developed to understand the rate of convergence in this case. We leave this for future work. Note also that for some values of $\ell$, we have $\rho(J(\mu_\alpha)) > \rho(\hat P_\alpha)$, which indicates that IAD would converge more slowly than the power method.

\begin{figure}[h]
  \label{fig: slightly irreversible 1d slow modes}
  \caption{Left: The steady-state distribution $\mu_\alpha$ of $P_\alpha$ for $\alpha=0.05$. For comparison, we have included the discrete Boltzmann distribution $\mu$ as well. Right: The first three left eigenvectors of $P^{\ast,1/{\mu_\alpha}}_\alpha P_\alpha$ for $\alpha=0.05$.}
  \begin{center}
    \includegraphics{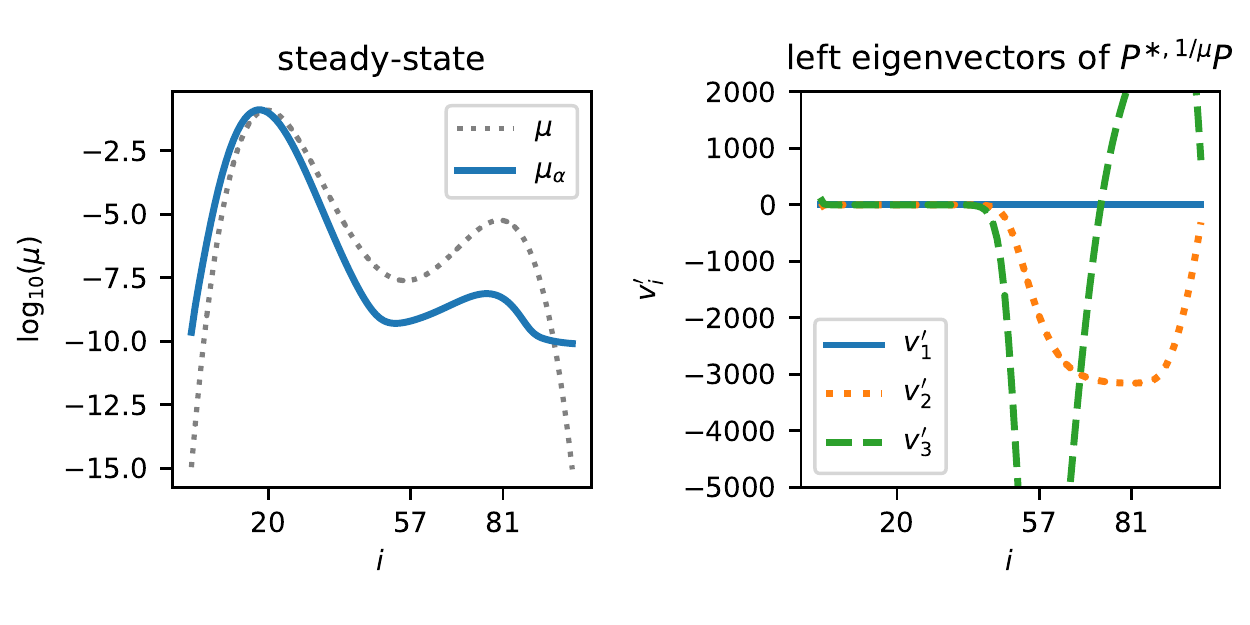}
  \end{center}
\end{figure}

\begin{figure}[h]
  \label{fig: slightly irreversible shift study}
  \caption{The spectral radius $\rho(J(\mu_\alpha))$, the norm upper bound~\eqref{eq: norm upper bound on rho}, and the angle upper bound~\eqref{eq: angle upper bound on rho} with $k=2$ for the irreversible, one-dimensional chain $P_\alpha$ with $\alpha=0.05$. Each of these numbers $x$ is very close to one for all values of $\ell$, so we report $-\log_{10}(1-x)$. The variable $\ell$ on the horizontal axis relates to the definition of the coarse states, cf.~\eqref{eq: shifted coarse states}.}
  \begin{center}
    \includegraphics{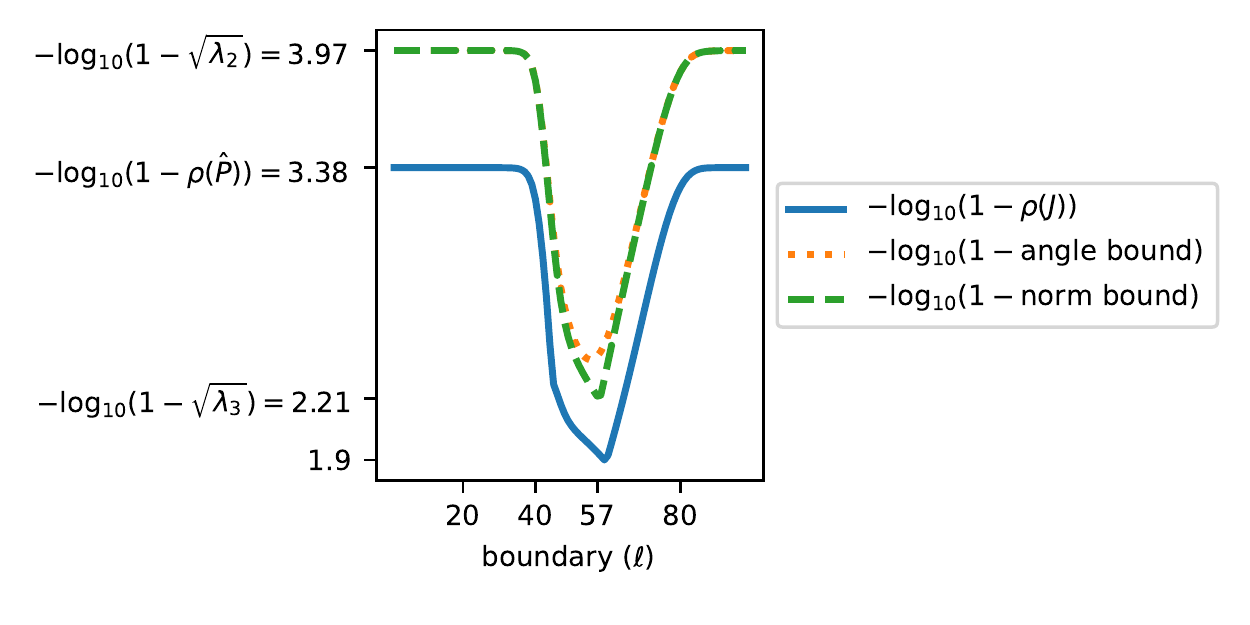}
  \end{center}
\end{figure}   


\begin{figure}[h]
  \label{fig: moderately irreversible shift study}
  \caption{The spectral radius $\rho(J(\mu_\alpha))$, the norm upper bound~\eqref{eq: norm upper bound on rho}, and the angle upper bound~\eqref{eq: angle upper bound on rho} with $k=2$ for the irreversible, one-dimensional chain $P_\alpha$ with $\alpha=0.15$. Each of these numbers $x$ is very close to one for all values of $\ell$, so we report $-\log_{10}(1-x)$. The variable $\ell$ on the horizontal axis relates to the definition of the coarse states, cf.~\eqref{eq: shifted coarse states}.}
  \begin{center}
    \includegraphics{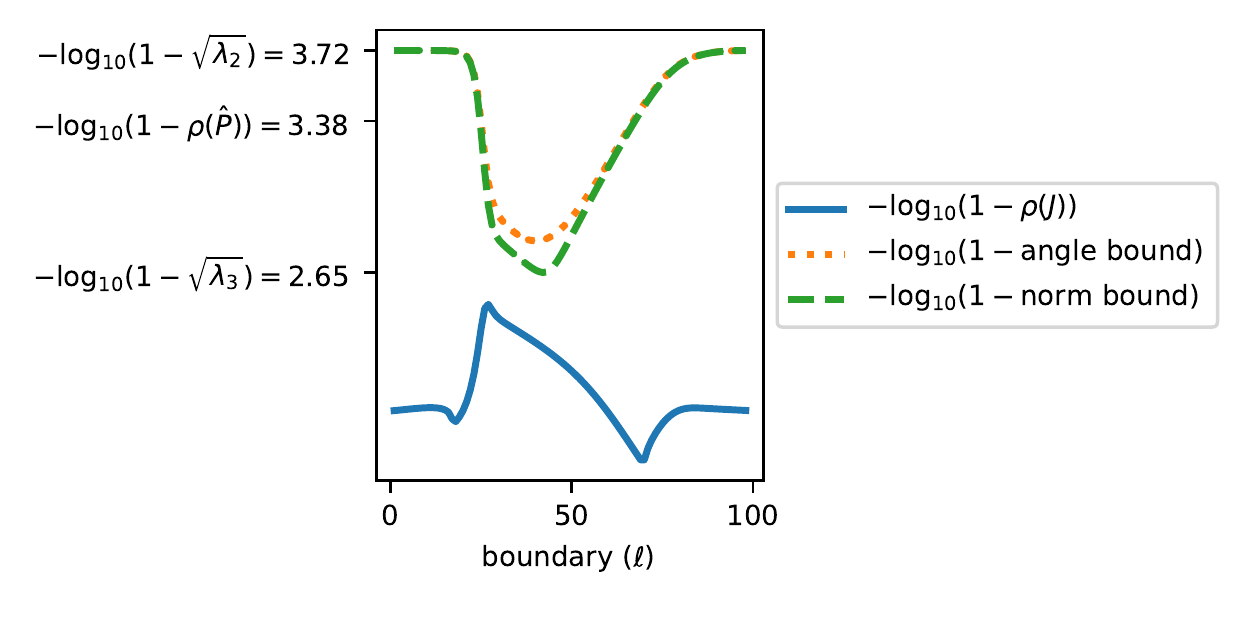}
  \end{center}
\end{figure}

\begin{table}[h]
  \label{tab: rates of convergence as function of alpha}
  \caption{The spectral radius $\rho(\hat P_\alpha)$ for $\alpha = 0, 0.05, 0.15$. Note that for $\alpha=0$, $P_\alpha$ is simply the reversible chain defined in Section~\ref{subsec: 1d computations}. Observe that for $\alpha =0.15$, the chain converges very quickly compared with $\alpha=0$. }
  \begin{tabular}{lll}
    \toprule
    $\alpha$ & $\rho(\hat P_\alpha)$ & $-\log_{10}(1- \rho(\hat P_\alpha))$ \\
    \midrule
    $0$ & $0.999992$ & $5.09$ \\
    $0.05$ & $0.999581$ & $3.38$ \\
    $0.15$ & $0.989564$ & $1.98$\\
    \bottomrule
  \end{tabular}
\end{table}

\subsection{IAD for a Metastable Chain on a Two-Dimensional Grid}
\label{subsec: 2d computations}

We now consider a metastable chain on a two-dimensional grid. Molecular models and other models in computational chemistry usually involve stochastic processes on spaces having thousands or millions of dimensions. Of course, when the dimension is so high, one cannot expect to cover space by a uniform grid of coarse states. Therefore, in many sampling strategies, one chooses a low-dimensional coordinate to discretize. We demonstrate for a model two-dimensional system that one can attain a significant reduction in the rate of convergence by discretizing a single variable into a small number of coarse states. 

Define $V: \Real^2 \rightarrow \Real$ by
\begin{equation*}
  \begin{split}
    V(x,y) &= 3 \exp \left (-x^2 -  \left (y-\frac13 \right  )^2\right) - 3 \exp  \left (-x^2 -  \left  (y^2 - \frac53 \right)^2\right) \\ &\qquad - 5 \exp  \left (- \left (x-1 \right)^2-y^2\right) - 5 \exp  \left (- \left (x+1\right)^2-y^2\right).
  \end{split}
\end{equation*}
See Figure~\ref{fig: park potential contour plot}.
This simple potential function was proposed for a study of reaction rates in~\cite{park_reaction_2003}. Let $\mu$ be the discrete Boltzmann distribution on a two-dimensional grid with $V$ as above and with $[a,b] \times [c,d] = [-1.7,1.7] \times [-1.7,2]$, $N=50$, and $T=1/4$. Let $P$ be the corresponding reversible Markov chain. The definitions of the discrete Boltzmann distribution and the reversible dynamics are analogous to the one-dimensional case. See Appendix~\ref{apx: 2d chain definition} for details. We report the largest $5$ eigenvalues of $P$ in Table~\ref{tab: reversible 2-d eigenvalues}, and we display the left eigenvectors $v'_2$ and $v'_3$ in Figures~\ref{fig: park potential v2} and~\ref{fig: park potential v3}. 

\begin{figure}[h]
  \label{fig: park potential contour plot}
  \caption{A Contour map of the potential function $V$. Note the two global minima on the left and right and the local minimum along a path betweeen the other two at the top.}
  \begin{center}
    \includegraphics{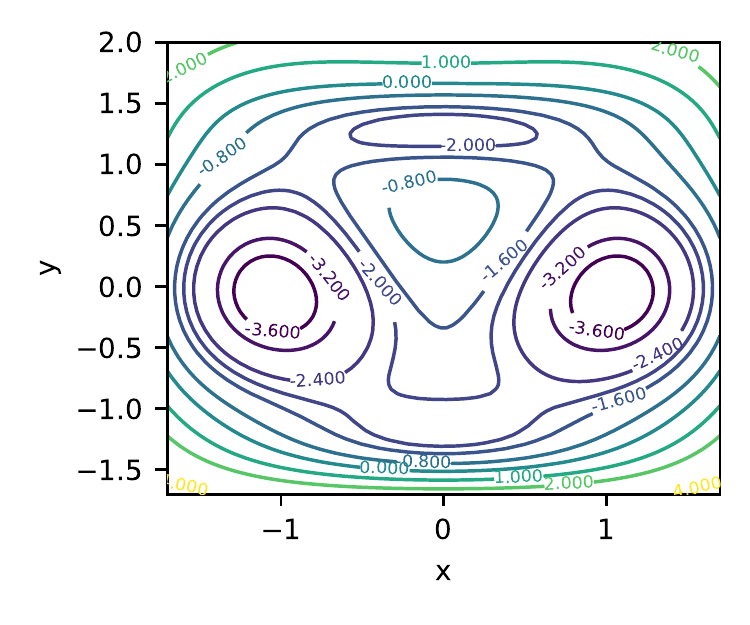}
  \end{center}
\end{figure}

Define the coarse states
\begin{equation}
  \label{eq: 1d grid strata}
 S_I = \left \{ (i,j)\in \{1, \dots, N\}^2: \left \lfloor \frac{i}3 \right \rfloor = j \right \} \text{ for } I =0,1,2.
\end{equation}
Here, $P$ is a Markov chain on the state space $\{1, \dots, 50 \} \times \{1, \dots, 50\}$, and the coarse states discretize only the first variable, not the second. The outlines of the coarse states are visible in the left panel of Figure~\ref{fig: park potential v2}. We report $\rho(J(\mu))$, the norm upper bound~\eqref{eq: norm upper bound on rho}, and the angle upper bound~\eqref{eq: angle upper bound on rho} for $k=2$ and $k=3$ for this choice of coarse states in Table~\ref{tab: reversible 2-d bounds and spectral radius}. Note that $\sin^2(\theta) \approx 0.002$ is quite small. We display the eigenvector $v'_2$ and its best approximation in the $\ell^2(\mu)$-norm in Figure~\ref{fig: park potential v2}. Although the two vectors may not appear to be aligned, they are in fact very close in the $\ell^2(\mu)$-norm since the regions where they differ have very small probability under $\mu$ and the maximum size of the difference between the vectors is not large in comparison to the very small probability. Observe that in this case, even with very few coarse states that discretize only a single dimension, we have in effect eliminated the largest eigenvalue $\sqrt{\lambda_2}$, and the rate of convergence of IAD is essentially equal to $\sqrt{\lambda_3}$. Note that $\sin^2(\theta)$ is large for $k=3$ in this case, so $v'_3$ is not well-approximated by a vector that is constant on the coarse states, cf.\@ Figure~\ref{fig: park potential v3}.

\begin{figure}[h]
  \label{fig: park potential v2}
  \caption{Right: The second left eigenvector $v'_2$ of $P^{\ast,1/\mu} P$. Left: The best approximation $\Pi(\mu)^\t v'_2$ of $v'_2$ in the $\ell^2(\mu)$-norm by a function that is constant on the one-dimensional grid of coarse states~\eqref{eq: 1d grid strata}. Note that this figure could also be interpreted as a depiction of the coarse states themselves.}
  \begin{center}
    \includegraphics{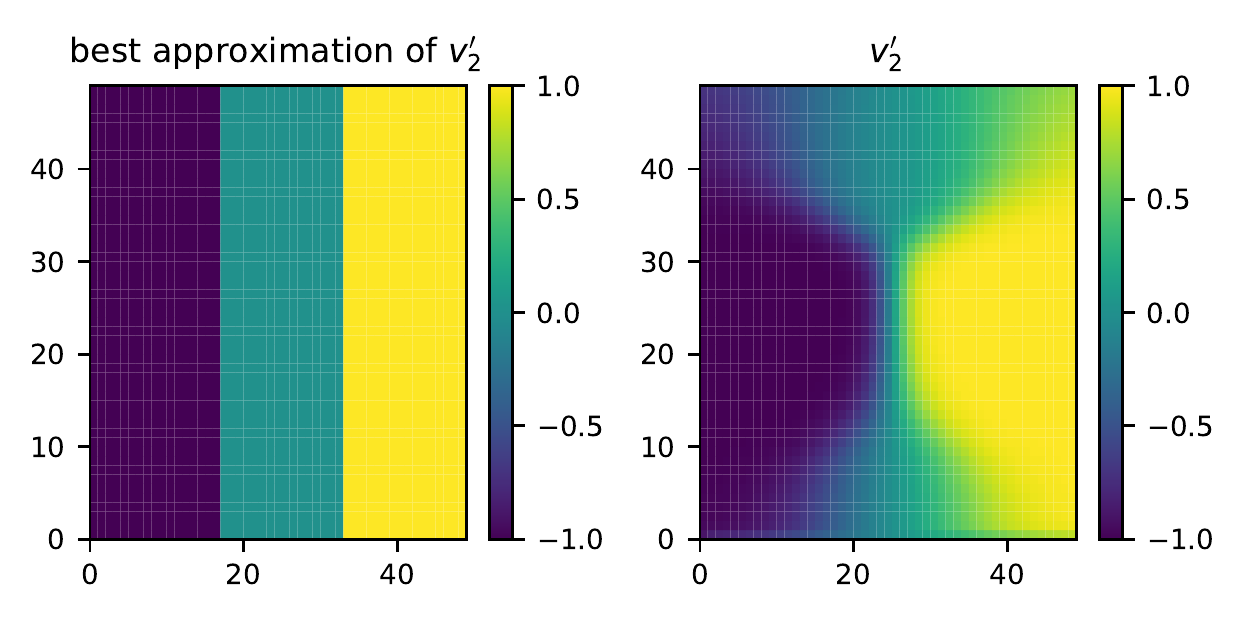}
  \end{center}
\end{figure}

\begin{table}[h]
  \label{tab: reversible 2-d eigenvalues}
    \caption{The largest five eigenvalues of $P^{\ast, 1/\mu} P$ for the reversible, two-dimensional chain $P$ of Section~\ref{subsec: 2d computations}. We report $\sqrt{\lambda_k}$ instead of $\lambda_k$, since it is $\sqrt{\lambda_k}$ that appears in Theorem~\ref{thm: estimate of rate of convergence with spectrum and angle involved}.}
  \begin{tabular}{lll}
    \toprule
    $k$ & $\sqrt{\lambda_k}$ \\
    \midrule
    $1$ & $1$ \\ 
    $2$ & $0.999997$ \\ 
    $3$ & $0.999488$ \\
    $4$ & $0.997511$ \\
    $5$ & $0.994219$ \\
          \bottomrule
  \end{tabular}
\end{table}

\begin{table}[h]
  \label{tab: reversible 2-d bounds and spectral radius}
  \caption{
  The spectral radius $\rho(J(\mu))$, the norm upper bound~\eqref{eq: norm upper bound on rho}, and the angle upper bound~\eqref{eq: angle upper bound on rho} together with $\sin^2(\theta)$ for $k=2,3$. The first column is for the one-dimensional grid~\eqref{eq: 1d grid strata} and the second column is for the six-by-six, two-dimensional grid~\eqref{eq: 2d grid strata}}
  \begin{tabular}{lllll}
    \toprule
    & 1-d Grid & 2-d Grid \\
    \midrule
    $\rho(J(\mu))$ & $0.999410$ & $0.987327$ \\
    norm bound  & $0.999410$ & $0.987327$ \\
    $\sin^2(\theta)$, $k=2$ & $0.002382$ & $0.000143$ \\
    angle bound, $k=2$ &  $0.999650$ & $0.999502$\\
    $\sin^2(\theta)$, $k=3$ & $0.854170$ & $0.031745$ \\
    angle bound, $k=3$ &  $0.999997$ & $0.999920$ \\
    \bottomrule
  \end{tabular}
\end{table}

We now test a six-by-six, two-dimensional grid of coarse states that refines the one-dimensional grid of three coarse states used above. For all $I,J \in \{0, \dots, 5\}$, we let
\begin{equation}
  \label{eq: 2d grid strata}
  S_{IJ} = \left  \{(x,y) \in \{1, \dots, N\}^2 : \left \lfloor \frac{x}6 \right \rfloor = I \text{ and } \left \lfloor \frac{y}6 \right \rfloor = J \right \}.
\end{equation}
One can see the outlines of some of these coarse states in the left panel of Figure~\ref{fig: park potential v3}. We report $\rho(J(\mu))$ and the upper bounds for this choice of coarse states in Table~\ref{tab: reversible 2-d bounds and spectral radius}. Note that both the spectral radius and the upper bounds are smaller, as expected for a refinement of the coarse states given the discussion following Corollary~\ref{cor: estimate of rate in reversible case}. For our six-by-six grid, $\sin^2(\theta)$ for $k=3$ is much smaller than for the one-dimensional grid. However, it is not small enough that the angle upper bound for $k=3$ is actually lower than for $k=2$. In fact, the interpolation between $\sqrt{\lambda_{k+1}}$ and $\sqrt{\lambda_2}$ in~\eqref{eq: angle upper bound on rho} is very steep, and one needs a very small value of $\sin^2(\theta)$ for the upper bound to approximate $\sqrt{\lambda_{k+1}}$. Nonetheless, refining the set of coarse states reduces $\rho(J(\mu))$ significantly, and this is correctly predicted by the norm upper bound~\eqref{eq: norm upper bound on rho}.

\begin{figure}[h]
  \label{fig: park potential v3}
  \caption{Right: The third left eigenvector $v'_3$ of $P^{\ast,1/\mu} P$. Left: The best approximation $\Pi(\mu)^\t v'_3$ of $v'_3$ in the $\ell^2(\mu)$-norm by a function that is constant on the two-dimensional grid of coarse states~\eqref{eq: 1d grid strata}.}
  \begin{center}
    \includegraphics{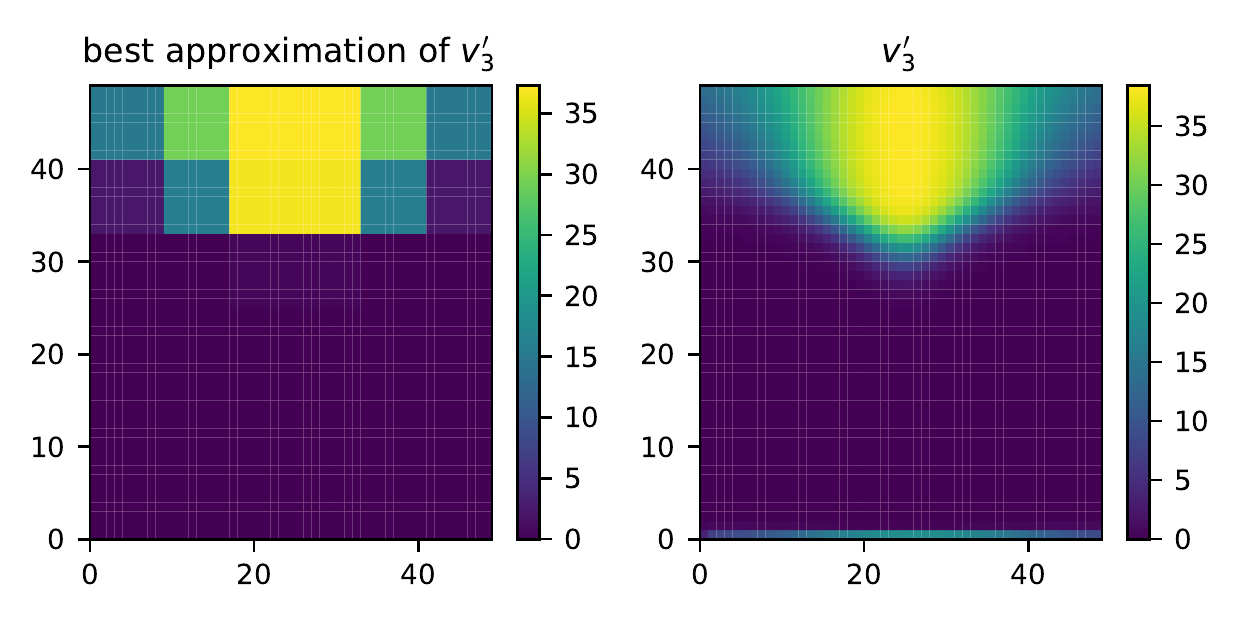}
  \end{center}
\end{figure}

\section{Conclusion}

Our work here was motivated by a desire to understand the robustness and efficiency of methods such as nonequilibrium umbrella sampling (NEUS)~\cite{warmflash_umbrella_2007}, exact milestoning~\cite{bello-rivas_exact_2015}, and injection measures~\cite{earle_convergence_2022} for calculating nonequilibrium (or equilibrium) steady-states in statistical physics. We have studied IAD as a simple model of this class of methods. We explain why it may be possible to use methods similar to IAD to efficiently compute steady-states of molecular models and how one might choose the coarse states in practice to optimize efficiency. For reversible processes, we conclude that one should choose coarse states so that the leading left eigenvectors of $P$ are well-approximated in the $\ell^2(\mu)$-norm by vectors that are constant on the coarse states. Since error is measured in the $\ell^2(\mu)$-norm, regions of low probability will not have a significant influence on the approximation quality unless some of the leading eigenvectors are concentrated in those regions. This means that in some cases a very na\"ive choice of coarse states can be efficient. For irreversible processes, our conclusions are similar but not so definite. For some very irreversible processes, our upper bounds do not yield much information about the dependence of the asymptotic rate of convergence on the choice of strata.
Although our primary interests lie in statistical physics, our results are general, and we hope others will apply them to understand the performance of IAD in other contexts. 

Our work does not address all important points. We show only local convergence, not global. We focus primarily on estimates of the asymptotic rate of convergence, ignoring preasymptotic phenomena. We recall that NEUS and similar methods are stochastic evolving particle systems that approximate IAD; we do not consider issues related to the particle approximation. We leave these points for future work.

\appendix

\section{Computing the Coarse Steady-State}
\label{apx: computing coarse distribution}

We compute the coarse steady-state $z(C(\mu^k))$ by the following algorithm.

\begin{algorithm}
  Assume that $C(\mu^k)$ is irreducible as guaranteed by Lemma~\ref{lem: well-posedness}. The user must specify an error tolerance $\tau>0$ and an exponent $k \in \mathbb{N}$.  In our numerical experiments in Sections~\ref{subsec: 1d computations} and~\ref{subsec: 2d computations}, we take $\tau = 10^{-9}$ and $k = 2^{15}$. We compute $z(C(\mu^k))$ by the following procedure:
  \begin{enumerate}
  \item  Set $\bar C = \frac12(I + C(\mu^k))$. Use the algorithm in~\cite{golub_using_1986} to compute an initial approximation $\tilde z$ to the steady-state $z(C(\mu^k))$ from the $QR$-factorization of $I-\bar C$.
  \item Refine the initial approximation $\bar z$ using a power method: 
    \begin{enumerate}
     \item 
      Set $z^{\text{old}} = \bar z$.
    \item
      Calculate
      \begin{equation*}
        z^{\text{new}} = \bar C^k z^{\text{old}}.
      \end{equation*}
     \item 
      If
      \begin{equation*}
        \max_{i =1, \dots, n} \frac{\lvert z^{\text{new}}_i  - C z^{\text{new}}_i \rvert }{z^{\text{new}}_i} < \tau \text{ and } \frac{\lvert z^{\text{new}}_i  - z^{\text{old}}_i \rvert }{z^{\text{old}}_i} < \tau
      \end{equation*}
      then return $z^{\text{new}}$ to the user. Otherwise, set $z^{\text{old}} = z^{\text{new}}$ and go to step~(b) above.
    \end{enumerate}
  \end{enumerate}
\end{algorithm}

In practice, we find that direct methods (like the algorithm in~\cite{golub_using_1986}) for calculating the coarse-steady state are often not sufficiently accurate. When direct methods produce an inaccurate result, applying a few power method iterations has usually produced a much better estimate of the steady-state. 

\section{Well-posedness of IAD}

Here, we prove that IAD is well-posed under our assumptions, and we give some examples to illustrate what can go wrong when our assumptions do not hold. 

\subsection{Proof of Lemma~\ref{lem: well-posedness}}
\label{apx: well-posedness}

We show that IAD is well-posed if $P$ is irreducible and $\mu^0 >0$. 

\begin{proof}
  First, we show that if $P$ is irreducible and $\nu >0$, then $C(\nu)$ is irreducible.
   We prove the contrapositive, showing that if $\nu >0$ and $C(\nu)$ is reducible, then $P$ is reducible. If $\nu >0$ and $C(\nu)$ is reducible, then there is a partition of the coarse states
  \begin{equation*}
    \{1, \dots, n\}= A \cup B
  \end{equation*}
  into disjoint and nonempty sets $A$ and $B$ so that for all $a \in A$ and $b\in B$
  \begin{equation*}
    C(\nu)_{ab} = A P D(\nu)_{ab} =  \sum_{\substack{i \in S_a \\ j \in S_b}} \frac{\nu_i}{A\nu_a} P_{ji} = 0. 
  \end{equation*}
  Therefore, since $\nu >0$, $P_{ji} =0$ for all $i \in S_a$ and $j \in S_b$. Now define
  \begin{equation*}
    \mathcal{A} = \cup_{a \in A} S_a \text{ and } \mathcal{B} = \cup_{b \in B} S_b.
  \end{equation*}
  The sets $\mathcal{A}$ and $\mathcal{B}$ are a partition of $\Omega$ into disjoint and nonempty sets, and we have $P_{ji} =0$ for any $i \in \mathcal{A}$ and $j \in \mathcal{B}$. It follows that $P$ must be reducible. We conclude that if $\nu >0$ and $P$ is irreducible, then $C(\nu)$ must be irreducible. 

  Now we show that if $\mu^k >0$, then $\mu^{k+1} >0$. To verify that $\mu^{k+\frac12}>0$, we observe that $C(\mu^k)$ is irreducible when $\mu^k >0$ by the previous paragraph, and so the steady-state $z(C(\mu^k))$ is unique and positive by the Perron--Frobenius theorem. It follows that $\mu^{k+\frac12} = D(\mu^k) z(C(\mu_k))>0$. Moreover, when $P$ is irreducible, $\nu >0$ implies $P \nu >0$. To see this, observe that if $\nu >0$, then for any $i \in \Omega$ with $P \nu_i =0$, $P_{ij} =0$ for all $j \neq i$. Thus, $P$ is reducible if $P \nu$ is not positive. Therefore, $\mu^{k+1} = P \mu^{k + \frac12} >0$ if $\mu^k >0$, and by induction $\mu^0 >0$ implies $\mu^k >0$ for all $k \in \mathbb{N}$. This concludes the proof that the iterates $\mu^k$ are well-defined. 
\end{proof}

\subsection{Examples motivating our assumptions}
\label{apx: pathological examples}

We now give some pathological examples to motivate our assumptions that $\mu^0 >0$ and that $P^\t P$ is irreducible.

\subsubsection{Positive initial condition ($\mu^0 > 0$)} We assume $\mu^0 >0$, since if $\mu^0$ is not positive, then $C(\mu^0)$ may be reducible even when $P$ and $P^\t P$ are irreducible and aperiodic. 
For example, consider the chain on $\Omega=\{1,2,3\}$ with transition matrix
\begin{equation*}
  P=
  \begin{pmatrix}
    0 &\frac13 &0 \\
    1 &\frac13 &1 \\
    0 &\frac13 &0
  \end{pmatrix}.
\end{equation*}
Here, $P$ and $P^\t P$ are irreducible and aperiodic. 
Now define the coarse states
\begin{equation*}
  S_1 = \{1,2\} \text{ and } S_2 = \{3\}.
\end{equation*}
For $\mu^0 = \frac12 (\delta_1 + \delta_3) = (\frac12, 0, \frac12)^\t$, we have
\begin{equation*}
 C(\mu^0) =  A P D (\mu^0) =
  \begin{pmatrix}
    1 &1 \\
    0 &0 \\
  \end{pmatrix},
\end{equation*}
which is reducible.

\subsubsection{$P^\t P$ irreducible} We assume that $P^\t P$ is irreducible to prove local convergence of IAD. It was observed in~\cite[Example 2]{marek_note_2006} that IAD is not locally convergent for
\begin{equation*}
  P =
  \begin{pmatrix}
    0 & 1 & 0 & \frac12 \\
    \frac12 & 0 & 0 & 0 \\
    \frac12 & 0 & 0 & \frac12 \\
    0 & 0 & 1 & 0
  \end{pmatrix}
\end{equation*}
with the strata $S_1 = \{1,2\}$ and $S_2 = \{3, 4\}$. 
However, $P$ is irreducible and aperiodic, so the Markov chain with transition matrix $P$ is convergent. The reader will note that $P^\t P$ is reducible, so this is an example where the power method is convergent but not a strict contraction in $\ell^2(1/\mu)$.

\section{Proofs of results stated in Section~\ref{subsec: power method}}

\subsection{Proof of Lemma~\ref{lem: time reversal}}
\label{apx: time reversal}

We begin with a proof of Lemma~\ref{lem: time reversal}, which shows that the time reversal of $P$ is its adjoint $P^{\ast, 1/\mu}$.

\begin{proof}[Proof of Lemma~\ref{lem: time reversal}]
 For simplicity, we write $\langle, \rangle$ for $\langle, \rangle_{1/\mu}$ and $P^\ast$ for $P^{\ast, 1/\mu}$. Note that for any $x, y \in \Real^N$, we have 
   \begin{align*}
     \langle x, P y \rangle &= \langle x, \diag(1/\mu) P y \rangle_{\1} \\
                                    &=  \langle \diag(\mu) P^\t \diag(1/\mu) x, \diag(1/\mu)  y \rangle_{\1} \\
                                    &=  \langle \diag(\mu) P^\t \diag(1/\mu) x,  y \rangle,
   \end{align*}
   so
   \begin{equation}\label{eqn: formula for P ast}
     P^\ast =  \diag(\mu) P^\t \diag(1/\mu).
   \end{equation}
   Therefore, $P^\ast_{ij} = \frac{P_{ji}\mu^j}{\mu_i}$, which is exactly the time-reversal of $P$. See~\cite[{Theorem~1.9.1}]{norris_markov_1998} for the definition of the time-reversal and its properties.  (Remember that $P$ is \emph{column} stochastic, so the time reversal here takes a slightly different form than for the row stochastic matrices of~\cite{norris_markov_1998}.) The time-reversal $P^\ast$ is column stochastic and has the same invariant distribution $\mu$ as $P$, since time-reversals always have these properties.
\end{proof}

\subsection{Proof of Lemma~\ref{lem: convergence of power method}}
\label{apx: power method}

We now prove convergence of the power method when both $P$ and $P^\t P$ are irreducible. The condition $P^\t P$ implies that the stochastic matrix $P^{\ast, 1/\mu} P$ is irreducible, and this is the essential fact in our proof that the power method is strictly contracting in the $\ell^2(1/\mu)$-norm.

\begin{proof}[Proof of Lemma~\ref{lem: convergence of power method}]
  To simplify notation, for any $M \in \Real^{N \times N}$, we write $M^\ast$ for $M^{\ast, 1/\mu}$ and $M^\t$ for the transpose. We let $\lVert \cdot \rVert$ and $\langle , \rangle$ denote the $\ell^2(1/\mu)$-norm and inner product.
 Since $P$ is column stochastic and $\nu^0$ is a probability vector, $\nu^k$ is a probability vector for all $k\in \mathbb{N}$, so $\1^\t \nu^k=1$. Therefore,
  \begin{align*}
    \nu^{k+1} - \mu &= P \nu^k - \mu \\
                    &= (P - \mu \1^\t) \nu^k \\
                    &= (P - \mu \1^\t ) (\nu^k - \mu) \\
                    &= \hat P (\nu^k - \mu).
  \end{align*}
   Note that the third equality above follows since $P\mu = \mu$ and $\mu$ is a probability vector.

   We now observe that for any $M \in \Real^{ N \times N}$, we have
 \begin{equation}
   \label{eq: norm in terms of spectrum}
     \lVert M \rVert = \rho (M^\ast M)^\frac12 = \lVert M^\ast M \rVert^\frac12.
   \end{equation}
   In particular, $\lVert \hat P \rVert = \rho (\hat P^\ast \hat P)^\frac12 = \lVert \hat P^\ast \hat P \rVert^\frac12$.
   The analogous result
   \begin{equation*}
     \lVert M \rVert_\1 = \rho (M^\t M)^\frac12 = \lVert M^\t M \rVert_{\1}^\frac12
   \end{equation*}
   for the uniformly weighted $\ell^2(\1)$-inner product is well-known, and the standard proof generalizes to any inner product space. Therefore, it will suffice to show that $\rho (\hat P^\ast \hat P)^\frac12 < \sqrt{\lambda_2}$. Observe that
  \begin{align*}
    \hat P^\ast \hat P &= (P^\ast -\mu \1^\t )(P - \mu \1^\t) \\
                       &= P^\ast P - P^\ast \mu \1^\t - \mu \1^\t P + \mu \1^\t \\
    &= P^\ast P -   \mu \1^\t,
  \end{align*}
  so  $\rho (\hat P^\ast \hat P)= \lambda_2$ by the diagonalization~\eqref{eq: diagonalization} of $P^\ast P$. Thus, $\lVert \hat P \rVert = \sqrt{\lambda_2}$. 

  By~\eqref{eqn: formula for P ast}, $P^\ast P$ is irreducible if and only if $P^\t P$ is irreducible. Moreover, $\lVert \hat P \rVert = \sqrt{\lambda_2} <1$ if $P^\ast P$ is irreducible by the Perron--Frobenius theorem, cf.\@\cite[Section~8.3, pg.\@ 673]{meyer_matrix_2008}. Therefore, $P^\t P$ irreducible is a sufficient condition for $\lVert \hat P \rVert <1$.  To see that it is also a necessary condition, we will prove that  if $P^\t P$ is reducible, then $\lVert \hat P \rVert \geq 1$. If $P^\t P$ is reducible, then there exist nonempty, disjoint subsets $A$ and $B$ of $\Omega$ so that $P^t P_{ij} = 0$ for all $i \in A$ and $j \in B$. Since $P^\t P$ is symmetric, we also have $P^\t P_{ij}=0$ for $i \in B$ and $j \in A$. Therefore, $P^\ast P$ admits a block decomposition of the form
  \begin{equation*}
    P^\ast P = 
    \bordermatrix{
      ~ &A & B \cr
      A & P_1 & 0 \cr
      B & 0 & P_2},
  \end{equation*}
  where $P_1$ and $P_2$ are both column stochastic matrices. Let $v_1$ and $v_2$ be steady-state probability vectors of $P_1$ and $P_2$, respectively. Define $V_1 = (v_1, 0)^\t$ and $V_2 =(0,v_2)^\t$. We have $V_1 - V_2 \neq 0$, and 
  \begin{align*}
   \lVert \hat P (V_1 - V_2) \rVert^2 =  \lVert P (V_1 - V_2)  \rVert^2 = \langle P^\ast P (V_1 - V_2), (V_1 - V_2) \rangle = \lVert V_1 - V_2 \rVert^2, 
  \end{align*}
  which verifies $\lVert \hat P \rVert \geq 1$.
\end{proof}

\section{Proofs of results stated in Section~\ref{subsec: local convergence}}

\subsection{Proof of Lemma~\ref{lem: inverse formula for invariant distribution}}
\label{apx: inverse formula}

We begin with a proof of Lemma~\ref{lem: inverse formula for invariant distribution}, our reformulation of the steady-state eigenproblem $Q z(Q) = z(Q)$ as a linear system. 

\begin{proof}[Proof of Lemma~\ref{lem: inverse formula for invariant distribution}]
  Observe that $x=z(Q)$ solves
  \begin{equation}\label{eqn: first eqn in proof of inverse formula lemma}
    (I - Q + v w^\t) x = v w^\t z(Q),
  \end{equation}
  since $Qz(Q)=z(Q)$. If we can show that $x=z(Q)$ is the unique solution of~\eqref{eqn: first eqn in proof of inverse formula lemma}, then $I - Q + v w^\t$ is invertible and the result follows.

  Suppose to the contrary that $u\neq z(Q)$ also solves~\eqref{eqn: first eqn in proof of inverse formula lemma}. Then we have
  \begin{equation*}
    \1^\t (I - Q + v w^\t) u = \1^\t v w^\t u = \1^\t v w^\t z(Q), 
  \end{equation*}
  since $\1^\t (I-Q) =0$ for any column stochastic $Q$. Therefore, $w^\t u= w^\t z(Q)$  since we assume $\1^\t v \neq 0$. It follows, again by~\eqref{eqn: first eqn in proof of inverse formula lemma}, that
  \begin{equation*}
    (I-Q)u=0.
  \end{equation*}
  Moreover, since we assume $w^\t z(Q) \neq 0$, $u \neq z(Q)$ and $w^\t u= w^\t z(Q)$ imply
  \begin{equation*}
    u \notin \spn \{z(Q)\}.
  \end{equation*}

  Now recall that since $Q$ is irreducible, $z(Q) >0$ by the Perron--Frobenius theorem. Thus, for some $\eps >0$, $z(Q)+\eps u >0$, $z(Q)+\eps u \notin \spn\{z(Q)\}$, and
  \begin{equation*}
    Q(z(Q)+\eps u)= z(Q)+\eps u.
  \end{equation*}
  This contradicts uniqueness of the stationary distribution of $Q$, so $z(Q)$ is the unique solution of~\eqref{eqn: first eqn in proof of inverse formula lemma}, and $I - Q + v w^\t$ is invertible.
\end{proof}

\subsection{Proof of Lemma~\ref{lem: error propagation for coarse correction step}}
\label{apx: error propagation for coarse correction}

We prove our error propagation formula for the coarse correction step. 

\begin{proof}[Proof of Lemma~\ref{lem: error propagation for coarse correction step}]
  First, note that
  \begin{equation*}
    A (I - P + \mu \1^\t)  D(\nu)  = I  - C(\nu) + A\mu \1^\t
  \end{equation*}
  is invertible by Lemma~\ref{lem: inverse formula for invariant distribution}, since $C(\nu)$ is irreducible when $\nu >0$ by Lemma~\ref{lem: well-posedness}. 

  We prove formula~\eqref{eq: error after coarse step} for the error after the coarse correction. Using our reformulations~\eqref{eq: linear formulation of fine problem} and~\eqref{eq: linear formulation of coarse problem} of the steady-state problems, we have
  \begin{align*}
    \mu^{k+\frac12}  &= D(\mu^k) z(C(\mu^k)) \\
                     &= D(\mu^k) (I - C(\mu^k) + A\mu \1^\t)^{-1} A\mu \\
                     &= D(\mu^k) (A(I - P + \mu \1^\t)D(\mu^k))^{-1} A\mu \\
                     &= D(\mu^k) (A(I - P + \mu \1^\t)D(\mu^k))^{-1} A (I -  P + \mu\1^\t)\mu \\
                     &= S(\mu^k) \mu.
  \end{align*}
  Now since $D(\mu^k)\mu^k = \mu^k$, we have
  \begin{align*}
    S(\mu^k) \mu^k &= D(\mu^k) (A(I - P + \mu \1^\t)D(\mu^k))^{-1} A (I -  P + \mu\1^\t) \mu^k \\
                   &= D(\mu^k) (A(I - P + \mu \1^\t)D(\mu^k))^{-1} A (I -  P + \mu\1^\t) D(\mu^k) \mu^k \\
                   &= \mu^k.
  \end{align*}
  Therefore,
  \begin{equation*}
    \mu^{k+ \frac12} - \mu = (I - S(\mu^k))\mu = (I - S(\mu^k))(\mu^k-\mu), 
  \end{equation*}
  as desired.

 It remains to show that $S(\nu)$ is a projection on $\rg(D(\nu))$. For convenience, we write $S$ for $S(\nu)$, $D$ for $D(\nu)$, and $C$ for $C(\nu)$. 
  To see that $S$ is a projection, observe that
  \begin{align*}
    S^2 &= D [A(I - \hat P) D]^{-1} (A (I- \hat P) D) [A(I - \hat P) D]^{-1} A (I- \hat P) \\
        &= D [A(I - \hat P) D]^{-1} A (I- \hat P) \\
    &= S.
  \end{align*}
 To see that $\rg(S)=\rg(D)$, first observe that
  \begin{equation*}
     \rg ( A (I - \hat P)) = \rg(A) = \Real^n, 
   \end{equation*}
   since $(I - \hat P)$ is invertible by Lemma~\ref{lem: inverse formula for invariant distribution}.
   Therefore, since  $A (I - \hat P)  D = I - C + A\mu \1^\t$ is also invertible by Lemma~\ref{lem: inverse formula for invariant distribution}, 
   \begin{equation*}
     \rg ([A (I - \hat P)  D]^{-1} A (I - \hat P)) = \Real^n, 
   \end{equation*}
   and it follows that $\rg(S) = \rg(D)$.
\end{proof}

\subsection{IAD as an Algebraic Multigrid Method}
\label{apx: iad as multigrid}

Here, we explain that IAD is more-or-less an adaptive algebraic multigrid method. Roughly similar observations appear in~\cite{krieger_two-level_1995}. Recall that the steady-state $\mu$ is the unique solution $x$ of the linear system of equations
\begin{equation*}
  (I - P + \mu \1^\t) x = \mu, 
\end{equation*}
cf.~\eqref{eq: linear formulation of fine problem}.
Suppose that one were to try to solve this equation by algebraic multigrid with the restriction operator $A$. Typically, the prolongation operator would be the transpose of restriction. Suppose that instead one were to take an adjoint with respect to some non-uniform inner product. For example, let $\nu \in \Real^N$ be a positive probability vector, and define the prolongation operator to be the operator $A^\dagger \in \Real^{N \times n}$ satisfying
\begin{equation*}
  \langle A \eta, w \rangle_{1/A \nu} = \langle \eta, A^\dagger w \rangle_{1/\nu}
\end{equation*}
for all $w \in \Real^n$ and $\eta \in \Real^N$.
By Lemma~\ref{lem: A and D are adjoint to each other}, we have $A^\dagger = D(\nu)$. 

The coarse grid correction step for a multigrid method with restriction operator $A$ and prolongation $D(\nu)$ is
\begin{align*}
  \mu^{k+\frac12} &= \mu^k + D(\nu) (A(I- P + \mu \1^\t) D(\nu))^{-1} A(\mu - (I - P + \mu \1^\t)\mu^k) \\
                  &= \mu^k + D(\nu) (A(I- P + \mu \1^\t) D(\nu))^{-1} A(I - P + \mu \1^\t)(\mu - \mu^k) \\
  &= \mu^k - S(\nu) (\mu^k - \mu),
\end{align*}
where $S(\nu)$ is the coarse projection defined in Lemma~\ref{lem: error propagation for coarse correction step}. Here, the residual is $\mu - (I - P + \mu \1^\t)\mu^k$ and the coarse system matrix is $A(I- P + \mu \1^\t) D(\nu)$. Note that if one chooses $\nu = \mu^k$, the coarse grid correction is equivalent with the coarse correction step of IAD. After the coarse grid correction in a multigrid method, one computes several steps of a smoothing iteration, often using some version of the Jacobi or Gauss-Seidel method. In IAD, one performs a step of the power method, which corresponds to using $P$ as the smoothing matrix. 

Note that IAD is not a true multigrid method since the inner product and therefore the prolongation operator depend on the current approximation $\mu^k$ of $\mu$. Thus, IAD is nonlinear and the standard theory of multigrid methods does not apply. 

\subsection{Proof of Lemma~\ref{lem: properties of multigrid projection}}
\label{apx: properties of multigrid projection}

We prove that $\Pi(\nu)$ is an $\ell^2(\nu)$-orthogonal projection. We begin by showing that $A$ and $D(\nu)$ are adjoint in a certain sense.

\begin{lemma}
  \label{lem: A and D are adjoint to each other}
  For any $w \in \Real^n$ and $\eta \in \Real^N$, we have
  \begin{align*}
    \langle D(\nu) w, \eta \rangle_{1/\nu} = \langle w, A \eta \rangle_{1/A\nu}.
  \end{align*}
\end{lemma}

\begin{proof}
  We have
  \begin{align*}
    \langle D(\nu) w, \eta \rangle_{1/\nu} &= \sum_{j=1}^N \sum_{k=1}^n w_k \nu(j \vert S_k) \eta_j \frac{1}{\nu_j} \\
                                           &= \sum_{j=1}^N \sum_{k=1}^n w_k \frac{\1_{S_k}(j)}{A\nu_k} \eta j \\
                                           &= \sum_{k=1}^n w_k \left ( \sum_{j = 1}^N \1_{S_k} (j) \eta_j \right ) \frac{1}{A \nu_k} \\
    &=  \langle w, A \eta \rangle_{1/A\nu}.
  \end{align*}
\end{proof}

We now prove Lemma~\ref{lem: properties of multigrid projection}, which verifies that $\Pi(\nu)$ is an orthgonal projection.

\begin{proof}[Proof of Lemma~\ref{lem: properties of multigrid projection}]
  First, note that
  \begin{equation*}
    A D(\nu) = I \in \Real^{N \times N}.
  \end{equation*}
  Therefore,
  \begin{equation*}
    \Pi(\nu)^2  = D (\nu) (A D(\nu)) A = D(\nu) A  = \Pi(\nu),
  \end{equation*}
  so $\Pi(\nu)$ is a projection.

  To see that $\Pi(\nu)$ is orthogonal, note that by Lemma~\ref{lem: A and D are adjoint to each other}, for any $\eta, \kappa \in \Real^N$, 
  \begin{align*}
    \langle \Pi(\nu) \eta, \kappa \rangle_{1/\nu} &=  \langle D(\nu)A \eta, \kappa \rangle_{1/\nu} \\
                                                  &=  \langle A \eta, A \kappa \rangle_{1/A\nu} \\
                                                  &=  \langle \eta, D(\nu) A \kappa \rangle_{1/\nu} \\
    &=  \langle  \eta, \Pi(\nu)\kappa \rangle_{1/\nu}.
  \end{align*}
  Therefore, $\Pi(\nu)^{\ast, 1/\nu} = \Pi(\nu)$ so $\Pi(\nu)$ is an orthogonal projection in $\ell^2(1/\nu)$. 

  Finally, observe that $\Pi(\nu) = D(\nu) A$ is column stochastic, since both $A$ and $D(\nu)$ map probability vectors to probability vectors. We have $\Pi(\nu)\nu  = \nu$ directly from the definition of $\Pi(\nu)$, and $\Pi(\nu)$ is reversible since it is self-adjoint with respect to the $\ell^2(1/\nu)$-inner product, cf.~Lemma~\ref{lem: time reversal}.
\end{proof}

\subsection{Proof of Lemma~\ref{lem: eps norm of J}}
\label{apx: proof of lemma on eps norm of J}

We begin our proof of Lemma~\ref{lem: eps norm of J} by computing the block decomposition of $J(\mu)$ with respect to $\Pi(\mu)$. The essential facts are summarized in the following simple lemma.

\begin{lemma}
  \label{lem: products of s and pi}
  We have
  \begin{equation*}
    \Pi(\mu) S(\mu) = S(\mu) \text{ and } S(\mu) \Pi(\mu) = \Pi(\mu).
  \end{equation*}
\end{lemma}

\begin{proof}
  For convenience, we write $\Pi$ for $\Pi(\mu)$ and $S$ for $S(\mu)$. Since $S$ is a projection with $\rg(S)= \rg(\Pi)$ by Lemma~\ref{lem: error propagation for coarse correction step}, we have
  \begin{equation*}
    S \Pi = \Pi. 
  \end{equation*}
  Similarly, $\Pi S = S$. 
\end{proof}

As a consequence of Lemma~\ref{lem: products of s and pi}, we have the following block decomposition of $J(\mu)$. 

\begin{lemma}
  \label{lem: block decomposition of J}
  We have
  \begin{equation*}
    J(\mu) \Pi(\mu) = 0,
  \end{equation*}
  and
  \begin{equation*}
    \Pi(\mu) J(\mu)  = \Pi(\mu) - S(\mu).
  \end{equation*}
\end{lemma}

\begin{proof}
  For convenience, we write $J$ for $J(\mu)$, $\Pi$ for $\Pi(\mu)$, and $S$ for $S(\mu)$. By Lemma~\ref{lem: products of s and pi}, $S \Pi = \Pi$, so 
  \begin{equation}
    J \Pi = \hat P (I - S) \Pi = 0.
  \end{equation}
  We also have
  \begin{align}
    \Pi J &= \Pi \hat P (I - S) \nonumber \\
          &= \Pi \hat P - \Pi \hat P S \nonumber \\
          &= \Pi \hat P - D [A \hat P D (I - A \hat P D )^{-1}] A (I - \hat P)\nonumber \\
          &= \Pi \hat P - D [(I - A \hat P D )^{-1} - I]  A (I - \hat P)\nonumber \\
    &=  \Pi \hat P - D [(A(I - \hat P) D )^{-1} - I]  A (I - \hat P)\nonumber \\
          &= \Pi \hat P - S + \Pi (I - \hat P)\nonumber \\
          &=\Pi - S. \nonumber
  \end{align}
The fifth inequality above follows since $AD = I \in \Real^{n \times n}$. 
\end{proof}

By Lemma~\ref{lem: block decomposition of J},
\begin{equation}
  \label{eqn: j is j times i minus pi}
  J(\mu) = J(\mu)(I - \Pi(\mu)).
\end{equation}
This will be important in several of our proofs. It implies, for example, that
\begin{equation}
  \label{eqn: spectrum of j is spectrum of i - pi times j}
  \sigma(J(\mu)) = \sigma(J(\mu)(I - \Pi(\mu)) = \sigma((I- \Pi(\mu))J(\mu)), 
\end{equation}
since $\sigma(AB)= \sigma(BA)$ for any matrices $A, B \in \Real^{N \times N}$. We now proceed with the proof of Lemma~\ref{lem: eps norm of J}.

\begin{proof}[Proof of Lemma~\ref{lem: eps norm of J}]
  For convenience, we write $J$ for $J(\mu)$, $\Pi$ for $\Pi(\mu)$, and $S$ for $S(\mu)$.
  For $x \in \Real^N$,
  \begin{align*}
    \lVert J x \rVert_\eps^2 &= \lVert J (I-\Pi) x \rVert_\eps^2 \\
                             &=\lVert (I- \Pi) J (I-\Pi)x \rVert_{1/\mu}^2 + \eps \lVert \Pi J (I-\Pi) x \rVert_{1/\mu}^2 \\
                             &\leq \left (\lVert (I- \Pi) J \rVert_{1/\mu}^2 + \eps \lVert \Pi J \rVert^2_{1/\mu} \right ) \lVert (I- \Pi) x \rVert_{1/\mu}^2 \\
                             &\leq \left (\lVert (I- \Pi) J \rVert_{1/\mu}^2 + \eps \lVert \Pi J \rVert^2_{1/\mu} \right ) \lVert x \rVert_\eps^2  \\
    &= \left (\lVert (I- \Pi) J \rVert_{1/\mu}^2 + \eps \lVert \Pi - S \rVert^2_{1/\mu} \right ) \lVert x \rVert_\eps^2
  \end{align*}
  The first equality holds by~\eqref{eqn: j is j times i minus pi}. The inequality in the second to last step holds since $I- \Pi$ is an orthogonal projection and therefore $\lVert I - \Pi\rVert_{1/\mu} = 1$. The last equality holds by Lemma~\ref{lem: block decomposition of J}.
\end{proof}

\subsection{Proof of Theorem~\ref{lem: first estimate of norm of projection of J onto conditionals}}
\label{apx: proof of estimate of norm of projection of J}

We derive an upper bound on $\lVert (I - \Pi(\mu)) J(\mu) \rVert_{1/\mu}$.

\begin{proof}[Proof of Theorem~\ref{lem: first estimate of norm of projection of J onto conditionals}]
  For convenience, we write $J$ for $J(\mu)$, $\Pi$ for $\Pi(\mu)$, etc. We write $\lVert \cdot \rVert$ for $\lVert \cdot \rVert_{1/\mu}$, $\langle, \rangle$ for $\langle, \rangle_{1/\mu}$, and $\hat P^\ast$ for $\hat P^{\ast, 1/\mu}$. First, observe that since at least one coarse state contains more than one fine state, we have $I- \Pi \neq 0$.
 Since $J = J (I- \Pi)$ by~\eqref{eqn: j is j times i minus pi}, and since $I- \Pi$ is an orthogonal projection, we have  
  \begin{equation}\label{eq: only have to test on rg i minus pi}
    \lVert (I - \Pi) J  \rVert = \lVert (I-\Pi) J (I- \Pi) \rVert = \max_{\substack{x \in \rg(I- \Pi) \\ \lVert x \rVert=1}} \lVert (I - \Pi ) J x \rVert. 
  \end{equation}
  Now we claim 
  \begin{equation}\label{eq: first lyapunov function}
    \lVert (I - \Pi) x \rVert^2 - \lVert (I- \Pi) J x \rVert^2 = \langle (I- S)x, (I - \hat P^\ast \hat P) (I - S) x \rangle.
  \end{equation}
  for all $x \in \Real^N$.
  To prove this, observe that
  \begin{align*}
    \lVert (I - \Pi) x \rVert^2 - \lVert (I-\Pi) J x \rVert^2 &= \lVert (I-\Pi) x \rVert^2 - (\lVert J x \rVert^2 - \lVert \Pi J x \rVert^2) \\
     &= \lVert (I - \Pi) x \rVert^2 - \lVert \hat P (I-S) x \rVert^2 + \lVert (\Pi-S) x \rVert^2 \\
                                                     &= \lVert (I - \Pi) x \rVert^2 - \lVert \hat P (I-S) x \rVert^2 + \lVert - S(I-\Pi) x \rVert^2 \\
                                                     &= \lVert (I-\Pi) x - S(I - \Pi)x \rVert^2 - \lVert \hat P (I-S) x \rVert^2\\
                                                              &= \lVert (I-S) x \rVert^2 - \lVert \hat P (I-S) x \rVert^2 \\
    &= \langle (I- S)x, (I - \hat P^\ast \hat P) (I - S) x \rangle.
  \end{align*}
The first equality above follows since $\Pi$ is an orthogonal projection by Lemma~\ref{lem: properties of multigrid projection}, and therefore $\lVert \Pi J x \rVert^2 + \lVert (I - \Pi) J x \rVert^2 = \lVert J x \rVert^2$ by the Pythagoras theorem.  The second follows since $\Pi J = \Pi - S$ by Lemma~\ref{lem: block decomposition of J}. The third follows since $S\Pi = \Pi$ by Lemma~\ref{lem: products of s and pi}, so $\Pi - S = - S(I - \Pi)$. The fourth follows from the Pythagoras theorem, since $\rg(S) = \rg(\Pi)$ and therefore $ (I - \Pi) x$ and $- S (I - \Pi) x$ are orthogonal. The fifth follows since $(I-S)(I- \Pi) = I - S$ because $S \Pi  = \Pi$ by Lemma~\ref{lem: products of s and pi}.

  Combining the results of the last paragraph yields
  \begin{align*}
    \lVert (I - \Pi) J \rVert^2 &=  \max_{\substack{x \in \rg(I- \Pi) \\ \lVert x \rVert=1}} \lVert (I - \Pi ) J x \rVert^2 \\
                              &=  \max_{\substack{x \in \rg(I- \Pi) \\ \lVert x \rVert=1}} \lVert (I - \Pi) x \rVert^2 -\langle (I- S)x, (I - \hat P^\ast \hat P) (I - S) x \rangle \\
    &=  1- \min_{\substack{x \in \rg(I- \Pi) \\ x \neq 0}} \frac{\langle (I- S)x, (I - \hat P^\ast \hat P) (I - S) x \rangle}{\lVert (I - \Pi) x \rVert^2}.
  \end{align*}
  We now make the change of variable $z = (I-S)x$ in the above minimization problem. 
  By Lemma~\ref{lem: products of s and pi},
  \begin{equation*}
     (I - \Pi) x = (I - \Pi)(I - S) x = (I - \Pi) z. 
   \end{equation*}
   Moreover, since $(I - S)(I- \Pi) = (I-S)$ by Lemma~\ref{lem: products of s and pi}, we have $(I- S) \rg(I-\Pi) = \rg (I- S)$. Also, $\rg(I-S)$ has the same dimension as $\rg(I-\Pi)$ because $S$ and $\Pi$ are both projections on $\rg (D)$. Therefore, $(I- S) (\rg(I-\Pi)\setminus \{0\}) = \rg (I- S)\setminus \{0\}$. 
   It follows that
   \begin{align*}
     \lVert (I - \Pi) J \rVert^2 &= 1 - \min_{\substack{x \in \rg(I- \Pi) \\ x \neq 0}} \frac{\langle (I- S)x, (I - \hat P^\ast \hat P) (I - S) x \rangle}{\lVert (I - \Pi) x \rVert^2} \\
                                 &= 1 - \min_{\substack{z \in \rg(I-S) \\ z \neq 0}} \frac{\langle z, (I - \hat P^\ast \hat P) z \rangle}{\lVert (I - \Pi) z \rVert^2}, 
   \end{align*}
   which proves the first claim in the statement of this lemma. 

   To prove the second claim, we make the change of variable $w = (I-\hat P^\ast \hat P)^\frac12 z$. The square root  $(I-\hat P^\ast \hat P)^\frac12$ exists since $\rho( \hat P^\ast \hat P ) <1$ by Lemma~\ref{lem: convergence of power method}, and therefore $I - \hat P^\ast \hat P$ is symmetric and positive definite with respect to the $\ell^2(1/\mu)$-inner product. See the proof of Theorem~\ref{thm: estimate of rate of convergence with spectrum and angle involved} for a detailed proof of the existence of a similar square root. 
   We have
 \begin{align*}
   \lVert (I - \Pi) J \rVert^2 &= 1-  \min_{w \in  (I-\hat P^\ast \hat P)^\frac12 \rg(I-S)} \frac{\lVert w \rVert^2}{\lVert (I - \Pi) (I-\hat P^\ast \hat P)^{-\frac12} w \rVert^2} \\
                               &\leq 1 - \frac{1}{\lVert (I - \Pi) (I-\hat P^\ast \hat P)^{-\frac12} \rVert^2}.
 \end{align*}
 Note that the denominator here is non-zero because $I-\Pi\neq 0$ and $(I-\hat P^\ast \hat P)^{-\frac12}$ is invertible. 
 
 Now observe that
 \begin{align*}
   \lVert (I - \Pi) (I-\hat P^\ast \hat P)^{-\frac12} \rVert^2 &= \rho(( (I - \Pi) (I-\hat P^\ast \hat P)^{-\frac12})^\ast  (I - \Pi) (I-\hat P^\ast \hat P)^{-\frac12} ) \\
                                                               &= \rho ((I - \Pi)^\ast  ((I-\hat P^\ast \hat P)^{-\frac12})^\ast (I-\hat P^\ast \hat P)^{-\frac12} (I - \Pi)) \\
   &= \rho((I - \Pi)(I-\hat P^\ast \hat P)^{-1}(I - \Pi)) \\
                                                               &= \lVert (I - \Pi) (I-\hat P^\ast \hat P)^{-1} (I - \Pi) \rVert.
 \end{align*}
 The first equality above follows since for any $M \in \Real^{N \times N}$ we have $\lVert M \rVert^2 = \rho(M^\ast M)$ by~\eqref{eq: norm in terms of spectrum}.
 The third equality above follows since $\Pi$ is an orthogonal projection, so $\Pi^\ast = \Pi$. 
\end{proof}

\subsection{Proof of Theorem~\ref{thm: local convergence}}
\label{apx: local convergence}

We prove local convergence of IAD. 
\begin{proof}[Proof of Theorem~\ref{thm: local convergence}]
 By Theorem~\ref{lem: first estimate of norm of projection of J onto conditionals}, $\lVert (I- \Pi) J(\mu) \rVert_{1/\mu} <1$, and so for $\eps>0$ sufficiently small
  \begin{equation*}
    \lVert J(\mu) \rVert_\eps < \sqrt{\lVert (I- \Pi(\mu)) J(\mu) \rVert_{1/\mu}^2 + \eps \lVert \Pi(\mu) - S(\mu) \rVert^2_{1/\mu}} <1 
  \end{equation*}
  by  Lemma~\ref{lem: eps norm of J}.
  We observe that $J(\nu)$ is continuous as a mapping from the set of positive probability vectors to the space of operators on $\Real^N$ with any choice of norms. This is because $S(\nu)$ is continuous, and so is the operator product.  Therefore, there exist $r >0$  and $\zeta <1$ so that if  $\lVert \nu - \mu \rVert_{\eps} \leq r$, then 
\begin{equation*}
  \lVert J(\nu) \rVert_\eps <  \zeta.
\end{equation*}
The result follows.
\end{proof}

\section{Proofs of results stated in Section~\ref{sec: convergence rate}}

\subsection{Proof of Lemma~\ref{lem: asymptotic rate of convergence}}
\label{apx: asymptotic rate of convergence}

\begin{proof}

  Let $\lVert \cdot \rVert$ denote both the norm on $\Real^N$ and also the induced operator norm on $\Real^{N \times N}$. 
  By Lemma~\ref{lem: error propagation formula},
  \begin{align*}
    \limsup_{K \rightarrow \infty} \lVert \mu^K - \mu \rVert^{\frac1K}
    &= \limsup_{K \rightarrow \infty} \left \lVert \prod_{i=0}^{K-1} J(\mu^i) (\mu^0 - \mu)  \right \rVert^{\frac1K} \\
    & \leq \limsup_{K \rightarrow \infty} \left \lVert \prod_{i=0}^{K-1} J(\mu^i) \right \rVert^{\frac1K}. 
  \end{align*}
  By Gelfand's formula,
  \begin{equation*}
    \lim_{n \rightarrow \infty} \lVert J(\mu)^n \rVert^\frac1n = \rho(J(\mu)).
  \end{equation*}
  Thus, for any $\delta >0$ there exists an $N$ so that 
  \begin{equation*}
    \lVert J(\mu)^N \rVert^\frac1N \leq \rho(J(\mu)) + \delta.
  \end{equation*}
  Now write
  \begin{align*}
    & \log \left (\left \lVert \prod_{i=0}^{K-1} J(\mu_i) \right \rVert^\frac1K \right ) \\
    &\qquad \leq \frac1K \left \{ \sum_{i=0}^{\lfloor (K-1)/N \rfloor} \log \left \lVert \prod_{j=(i-1)N}^{iN-1} J(\mu_j) \right \rVert + \log \left  \lVert \prod_{j=\lfloor (K-1)/N \rfloor N}^{K-1} J(\mu_j) \right \rVert \right \} \\
    &\qquad \leq \frac{1}{\lfloor (K-1)/N \rfloor} \sum_{i=1}^{\lfloor (K-1)/N \rfloor} \frac{1}{N} \log \left \lVert \prod_{j=(i-1)N}^{iN-1} J(\mu_j) \right \rVert
      + \frac1K \sum_{j=\lfloor (K-1)/N \rfloor N}^{K-1} \log \left  \lVert  J(\mu_j) \right \rVert \\
   &\qquad =: A_K + B_K.
  \end{align*}
    Since $\lVert J(\mu_n) - J(\mu) \rVert \rightarrow 0$, we have
  \begin{equation*}
    \lim_{M \rightarrow \infty} \log \left \lVert \prod_{i=M+1}^{M+N} J(\mu_i) \right \rVert = \log \left \lVert J(\mu)^N \right \rVert \leq N \log( \rho(J(\mu))+\delta),
  \end{equation*}
  and so
  \begin{equation*}
    \lim_{K \rightarrow \infty} A_K = \frac1N \log \lVert J(\mu)^N \rVert \leq \log(\rho(J(\mu)) + \delta). 
  \end{equation*}
  Moreover, $\lim_{K\rightarrow \infty} B_K = 0$ and the result follows.
\end{proof}

\subsection{Proof of Theorem~\ref{thm: estimate of rate of convergence with spectrum and angle involved}}
\label{apx: estimate of rate with spectrum and angle}

In the proof of Theorem~\ref{thm: estimate of rate of convergence with spectrum and angle involved}, we use that the norms $\lVert \cdot \rVert_{1/\mu}$ and $\lVert \cdot \rVert_\mu$ are dual with respect to the unweighted $\ell^2(\1)$-inner product, and therefore $\lVert M \rVert_{1/\nu} = \lVert M^\t \rVert_\nu$ for any $M \in \Real^{N \times N}$. Similar results are well-known, and this might all be obvious to the reader, but we are unable to find a reference, so we offer a proof below. 

\begin{lemma}
  \label{lem: dual operator norms}
  For any $M \in \Real^{N \times N}$ and any positive $\nu \in \Real^N$, we have
  \begin{equation*}
    \lVert M \rVert_{1/\nu} = \lVert M^\t \rVert_\nu.
  \end{equation*}
\end{lemma}

\begin{proof}
  First, observe that
  \begin{align*}
    \lVert y \rVert_{1/\nu}^2 &= \langle y, \diag(1/\nu) y \rangle_\1 \\
                              &= \langle \diag(1/\nu) y, \diag(\nu)  \diag(1/\nu) y \rangle_\1 \\
                              &= \lVert \diag(1/\nu) y \rVert^2_\nu,
  \end{align*}
  so $\diag(1/\nu)$ is an isometry mapping $\ell^2(1/\nu)$ to $\ell^2(\nu)$.
  Therefore, we have
  \begin{align*}
    \lVert x \rVert_{1/\nu} &= \max_{\lVert y \rVert_{1/\nu} =1} \langle x, y \rangle_{1/\nu} \\
                            &= \max_{\lVert y \rVert_{1/\nu} =1} \langle x, \diag(1/\nu) y \rangle_{\1} \\
                            &= \max_{\lVert z \rVert_{\nu} =1} \langle x,  z \rangle_{\1}.
  \end{align*}
  The first equality above follows from the Cauchy--Schwarz inequality for $\ell^2(1/\nu)$. The last equality follows by making the change of variable $z=  \diag(1/\nu)$ and using that $\diag(1/\nu)$ is an isometry mapping $\ell^2(1/\nu)$ to $\ell^2(\nu)$.
  We now have  
  \begin{align*}
    \lVert M \rVert_{1/\nu} &= \max_{\lVert x \rVert_{1/\nu}=1} \lVert M x \rVert_{1/\nu} \\
                            &= \max_{\lVert z \rVert_{\nu}=1} \max_{\lVert x \rVert_{1/\nu}=1} \langle Mx, z \rangle_\1 \\
                            &= \max_{\lVert z \rVert_{\nu}=1} \max_{\lVert x \rVert_{1/\nu}=1} \langle x, M^\t z \rangle_\1 \\
                            &= \max_{\lVert z \rVert_{\nu}=1} \lVert M^\t z \rVert_\nu \\
                            &= \lVert M^\t \rVert_\nu,
  \end{align*}
  as claimed.
\end{proof}

We now prove Theorem~\ref{thm: estimate of rate of convergence with spectrum and angle involved}. 

\begin{proof}[Proof of Theorem~\ref{thm: estimate of rate of convergence with spectrum and angle involved}]
  For convenience, we write $\lVert \cdot \rVert$ for $\lVert \cdot \rVert_{1/\mu}$, $\Pi$ for $\Pi(\mu)$, and $M^\ast$ for $M^{\ast, 1/\mu}$.
  First, we observe that
  \begin{equation}
    \label{eqn: bound on rho by i minus pi times J}
    \rho(J) \leq \inf_{\eps >0} \lVert J \rVert_\eps =  \lVert ( I - \Pi ) J\rVert 
  \end{equation}
  by Lemma~\ref{lem: eps norm of J}.
  Also, observe that $I - \Pi \neq 0$, since at least one coarse state contains more than one fine state.
  
    We now estimate $\lVert (I - \Pi) (I - \hat P^\ast \hat P)^{-1} (I - \Pi) \rVert$. By Theorem~\ref{lem: first estimate of norm of projection of J onto conditionals}, an upper bound on this expression implies an upper bound on $\lVert (I- \Pi) J \rVert^2$, and therefore on $\rho(J)^2$ by~\eqref{eqn: bound on rho by i minus pi times J}.
  Since $\rho(\hat P^\ast \hat P) <1$ by Lemma~\ref{lem: convergence of power method},  $I - \hat P^\ast \hat P$ is symmetric positive definite with respect to the $\ell^2(1/\mu)$-inner product, and therefore so is $(I - \hat P^\ast \hat P)^{-1}$. It follows that there exists a square root                             
  \begin{equation*}
    L = (I - \hat P^\ast \hat P)^{-\frac12} = \sum_{k=2}^N (1-\lambda_k)^{-\frac12} v_i v_i^\t \diag(1/\mu),
  \end{equation*}
  which is also symmetric positive definite.
  We note that $L$ must commute with the spectral projector $Q$. Also,
  \begin{equation}
    \label{eq: norms of LQ and L(I-Q)}
    \lVert LQ \rVert = (1-\lambda_2)^{-\frac12} \text{ and } \lVert L(I-Q) \rVert =(1- \lambda_{k+1})^{-\frac12}.
\end{equation}
These facts follow directly from the above formula for $L$ in terms of the diagonalization of $P^\ast P$. 

 Since $(I-\Pi)$ is an orthogonal projection, $(I- \Pi)^\ast = I-\Pi$, and we have
  \begin{align*}
    \lVert (I - \Pi) (I - \hat P^\ast \hat P)^{-1} (I - \Pi) \rVert &= \lVert (I - \Pi) L^2 (I - \Pi) \rVert \\
                                                                    &= \lVert (L (I - \Pi))^\ast L (I - \Pi) \rVert \\
    &= \lVert L(I - \Pi) \rVert^2.
  \end{align*}
  The last equality above follows since for any $M \in \Real^{N\times N}$, $\lVert M^\ast M \rVert = \lVert M \rVert^2$ by~\eqref{eq: norm in terms of spectrum}.
We have
  \begin{align*}
    \lVert L (I- \Pi) \rVert^2 &= \max_{\lVert x \rVert=1} \lVert L (I- \Pi) x \rVert^2  \\
                               &= \max_{\lVert x \rVert=1} \lVert Q L (I- \Pi) x \rVert^2 + \lVert (I-Q )L(I- \Pi) x \rVert^2 \\
                               &= \max_{\lVert x \rVert=1} \lVert  LQ (I- \Pi) x \rVert^2 + \lVert L (I-Q )(I- \Pi) x \rVert^2 \\
    &= \max_{\lVert x \rVert=1} \lVert  LQ Q (I- \Pi) x \rVert^2 + \lVert L (I-Q )(I-Q)(I- \Pi) x \rVert^2 \\
                               &\leq \max_{\lVert x \rVert=1} \lVert LQ \rVert^2 \lVert Q(I- \Pi) x \rVert^2 + \lVert L (I-Q ) \rVert^2 \lVert (I-Q)(I- \Pi) x \rVert^2 \\
    &= \max_{\lVert x \rVert=1} \frac{1}{1-\lambda_1}  \lVert Q(I- \Pi) x \rVert^2 + \frac{1}{1- \lambda_{k+1}} \lVert (I-Q)(I- \Pi) x \rVert^2 \\
  \end{align*}
The second equality in the above display follows from the Pythagoras law, since $Q$ is an orthogonal projection. The third equality follows since $L$ and $Q$ commute, as explained in the previous paragraph. The last equality follows from our formulas in~\eqref{eq: norms of LQ and L(I-Q)} for the norms of $LQ$ and $L(I-Q)$.

Now observe that
\begin{equation*}
  \lVert (I - Q)(I- \Pi) x \rVert^2 +  \lVert  Q (I- \Pi) x \rVert^2 =  \lVert (I- \Pi) x \rVert^2 \leq 1,
\end{equation*}
again by the Pythagoras law.
Therefore,
\begin{equation}
  \label{eq: bound on i minus q times i minus pi}
  \lVert (I - Q)(I- \Pi) x \rVert^2 \leq 1 -   \lVert  Q (I- \Pi) x \rVert^2.
\end{equation}
By Lemma~\ref{lem: dual operator norms},
\begin{equation}
  \label{eq: dual angle formula}
  \lVert Q (I - \Pi) \rVert = \lVert (I - \Pi^\t) Q^\t \rVert_\mu = \sin(\theta).
\end{equation}
Combining~\eqref{eq: bound on i minus q times i minus pi} and~\eqref{eq: dual angle formula} with the results of the previous paragraph yields
\begin{align*}
   \lVert (I - \Pi) (I - \hat P^\ast \hat P)^{-1} (I - \Pi) \rVert
  &\leq \max_{\lVert x \rVert=1} \frac{1}{1-\lambda_1}  \lVert Q(I- \Pi) x \rVert^2 + \frac{1}{1- \lambda_{k+1}} \lVert (I-Q)(I- \Pi) x \rVert^2 \\
                                 &\leq \max_{0 \leq \alpha \leq \sin^2(\theta)}  \frac{1}{1-\lambda_1}  \alpha^2 + \frac{1}{1- \lambda_{k+1}} (1-\alpha^2) \\
    &= \frac{1}{1-\lambda_1}  \sin^2(\theta) + \frac{1}{1- \lambda_{k+1}} \cos^2(\theta).
\end{align*}
The result now follows by Theorem~\ref{lem: first estimate of norm of projection of J onto conditionals}. 
\end{proof}

\subsection{Proof of Theorem~\ref{thm: exact spectrum of J}}
\label{apx: proof of exact spectrum}

\begin{proof}
  For convenience, we write $J$ for $J(\mu)$, $\Pi$ for $\Pi(\mu)$, etc.
  Note that $I- \Pi \neq 0$ and $\Pi \neq 0$, since there is more than one coarse state and at least one coarse state contains more than one fine state. 
  Recall that $J$ has the same spectrum as $(I-\Pi)J$ by~\eqref{eqn: spectrum of j is spectrum of i - pi times j}.
  Observe that $0 \in \sigma((I-\Pi) J)$, since we have $J = J(I-\Pi)$ by~\eqref{eqn: j is j times i minus pi}, and so $Jx =0$ for any $x \in \rg (\Pi)$. Now suppose that
  \begin{equation}
    \label{eq: first step in exact spectrum formula}
    (I-\Pi)Jx=\lambda x
  \end{equation}
  for some $\lambda \neq 0$ and $x \neq 0$. 
  We will show that $x$ is an eigenvector of $(I - \Pi) (I - \hat P)^{-1} (I - \Pi)$ with eigenvalue $(1 - \lambda)^{-1}$. 
  (Note that $\lambda\ne1$ because $\rho((I-\Pi)J)<1$ by Theorem~\ref{lem: first estimate of norm of projection of J onto conditionals}, so $(1-\lambda)^{-1}$ is defined.) Since $\Pi J = \Pi - S$ by Lemma~\ref{lem: block decomposition of J}, we have 
  \begin{align*}
    (I-\Pi)-(I-\Pi)J &= I - \Pi - J + \Pi - S \\
                     &= I -S - \hat P(I-S) \\
                     &=(I-\hat{P})(I-S).
  \end{align*}
  Since $\lambda \neq 0$, the eigenvalue equation~\eqref{eq: first step in exact spectrum formula} implies that $x\in \rg(I-\Pi)$. Therefore, since $I - \Pi$ is a projection, we have $(I-\Pi)x = x$, and  
  \begin{equation*}
    (I-\hat{P})(I-S)x=(I- \Pi)x-(I-\Pi)Jx=(1-\lambda) (I - \Pi) x.
  \end{equation*}
  Multiplying on both sides above by $(I- \Pi)(I- \hat P)^{-1} $ yields
  \begin{align*}
    (1- \lambda)(I - \Pi) (I- \hat P)^{-1} (I- \Pi) x &= (I - \Pi) (I - S) x \\
                                              &= (I- \Pi) x\\
    &= x,
  \end{align*}
since we have $(I-\Pi) (I-S) = I - \Pi$ by Lemma~\ref{lem: products of s and pi}.
  Thus, $x$ is an eigenvector of $(I-\Pi)(I-\hat{P})^{-1} (I - \Pi)$ with eigenvalue $(1- \lambda)^{-1}$. We conclude that 
  \begin{equation*}
\sigma(J(\mu)) \subset \left(1-\frac{1}{\sigma((I-\Pi(\mu))(I-\hat{P})^{-1}(I-\Pi(\mu)))\setminus\{0\}}\right)\cup\{0\}.
  \end{equation*}

  It remains to prove the opposite inclusion. 
  Consider the operator
  \begin{equation*}
    M=(I-\hat{P})(I-S).
  \end{equation*}
  We have
  \begin{align*}
    \Pi M &= D A (I - \hat P) (I - D (A(I - \hat P) D)^{-1} A (I - \hat P)) \\
          &= DA (I - \hat P) - D A (I - \hat P) D (A(I - \hat P) D)^{-1} (I- \hat P)  \\
            &= 0,
  \end{align*}
  so $\rg(M) \subset \rg(\Pi)$.
  Moreover, since $I- \hat P$ is invertible, $\ker(M) = \rg(S) = \rg(\Pi)$. Therefore, when viewed as an operator on the range of $I-\Pi$, $M$ is invertible. 

  Now suppose that
  \begin{equation*}
    (I-\Pi)(I-\hat{P})^{-1}(I-\Pi)y=\alpha y
  \end{equation*}
  for some $\alpha \neq 0$ and $y \neq 0$. Then $y \in \rg(I-\Pi)$, and so we may write $y=Mx$ for some $x \in \rg(I-\Pi)$. Therefore,  
  \begin{align*}
    \alpha Mx&=(I-\Pi)(I-\hat{P})^{-1}(I-\Pi)Mx \\
             &=(I-\Pi)(I-\hat{P})^{-1}Mx \\
             &=(I - \Pi) (I- S)x \\
             &=(I-\Pi) x \\
             &=x,
  \end{align*}
  since $(I - \Pi) (I- S)  = I - \Pi$ and $x \in \rg(I-\Pi)$.
  Therefore, $\frac{1}{\alpha}$ is an eigenvalue of $M$ with eigenvector $x$. Now
  \begin{align*}
    (I-\Pi)J+M &= \left [ (I- \Pi) \hat P + (I - \hat P) \right ] (I -S) \\
               &= \left [ (I- \Pi) + \Pi(I - \hat P) \right ] (I -S) \\
               &= (I- \Pi) (I-S) + \Pi M \\
    &= (I- \Pi),
  \end{align*}
  since $\Pi M =0$ and $(I - \Pi) (I-S) = I- \Pi$ as explained above. Therefore,
  \begin{equation*}
    (I-\Pi)J=(I-\Pi)-M,
  \end{equation*}
  and $x\in \rg(I-\Pi)$ is an eigenvector of $(I-\Pi)J$ with eigenvalue $1-\frac{1}{\alpha}$. 
  Finally, $0$ is an eigenvector of $J$ because $Jx=0$ whenever $x\in \rg(\Pi)$. Therefore,
    \begin{equation*}
\sigma(J(\mu)) \supset \left(1-\frac{1}{\sigma((I-\Pi(\mu))(I-\hat{P})^{-1}(I-\Pi(\mu)))\setminus\{0\}}\right)\cup\{0\}.
  \end{equation*}
  
\end{proof}

\subsection{Proof of Corollary~\ref{cor: estimate of rate in reversible case}}
\label{apx: corollary for reversible case}

\begin{proof}[Proof of Corollary~\ref{cor: estimate of rate in reversible case}]
  For convenience, we write $\Pi$ for $\Pi(\mu)$, $J$ for $J(\mu)$ and $\lVert \cdot \rVert$ for $\lVert \cdot \rVert_{1/\mu}$. 
  Observe that $ (I- \Pi) (I - \hat P)^{-1} (I - \Pi)$ is self-adjoint as an operator on $\ell^2(1/\mu)$, since $P$ is reversible and $I- \Pi$ is an orthogonal projection, and so both $\hat P$ and $I- \Pi$ are self-adjoint. Therefore, $\sigma((I- \Pi) (I - \hat P)^{-1} (I - \Pi)) \subset \Real$. Since $P$ and $P^\t P$ are irreversible, we have $\rho(J) <1$ by Theorem~\ref{thm: local convergence}. Therefore by Theorem~\ref{thm: exact spectrum of J} we must have $\sigma((I- \Pi) (I - \hat P)^{-1} (I - \Pi)) \subset [0,\infty)$, because if any eigenvalue of $(I- \Pi) (I - \hat P)^{-1} (I - \Pi)$ were negative, then Theorem~\ref{thm: exact spectrum of J} would imply $\rho(J) >1$. It follows that
  \begin{align*}
    \rho(J) &= \max_{\lambda \in \sigma((I- \Pi) (I - \hat P)^{-1} (I - \Pi))} \left \lvert  1- \frac{1}{\lambda} \right \rvert \\
            &=  1 - \frac{1}{\rho((I- \Pi) (I - \hat P)^{-1} (I - \Pi))} \\
            &= 1 - \frac{1}{\lVert (I- \Pi) (I - \hat P)^{-1} (I - \Pi) \rVert }.
  \end{align*}
   The proof of the inequality is identical with the proof of Theorem~\ref{thm: estimate of rate of convergence with spectrum and angle involved}, except with $L = (I - \hat P)^{-\frac12}$ and $\sqrt{\lambda_i}$ in place of $\lambda_i$.
\end{proof}

\section{A Model of Overdamped Langevin Dynamics on a Two-Dimensional Grid}
\label{apx: 2d chain definition}

We define a family of Markov chains on a two-dimensional grid with properties similar to overdamped Langevin. Let $V : \Real^2 \rightarrow \Real$, $T >0$, $[a,b ] \times [c,d] \subset \Real^2$, and $N \in \mathbb{N}$. Define the discrete Boltzmann distribution $\mu \in \Real^{N \times N}$ by
\begin{equation}
  \label{eq: discrete Boltzmann}
  \mu_{ij} = Z^{-1} \exp \left (- \frac{ V \left ( a + \frac{b-a}{N} i, c + \frac{d-c}{N} i \right )}{T} \right ),
\end{equation}
where  $Z =\sum_{i =1}^N \exp \left (- \frac{ V \left ( a + \frac{b-a}{N} i, c + \frac{d-c}{N} i \right )}{T} \right )$.
Define a transition probability on $\Omega = \{1, \dots, N\}^2$ by
\begin{alignat*}{3}
P_{(i +k,j+\ell),(i,j)} &:= \frac{1}{4}\frac{\mu(i+k,j+\ell)}{\mu(i +k,j+\ell) + \mu(i,j)} &&\mbox{ for } k,\ell \in \{-1,1\} \text{ and } (i,j) \in \Omega \\
P_{(i,j),(i,j)} &:= 1- \sum_{k,\ell \in \{-1,1\}} P_{(i+k,j+\ell)} &&\mbox{ for } (i,j) \in \Omega \\
P_{(k,\ell),(i,j)} &:= 0 &&\mbox{ otherwise.} \\
\end{alignat*}
In the definition of $P$, we impose periodic boundary conditions, associating $(0,j)$ with $(N,j)$, $(i,1)$ with $(i,N+1)$, etc. Note that $P$  is in detailed balance with $\mu$.

\bibliographystyle{abbrv}
\bibliography{iad}

\begin{thebibliography}{10}

\bibitem{bello-rivas_exact_2015}
J.~M. Bello-Rivas and R.~Elber.
\newblock Exact milestoning.
\newblock {\em The Journal of Chemical Physics}, 142(9):094102, Mar. 2015.

\bibitem{bhatt_steady-state_2010}
D.~Bhatt, B.~W. Zhang, and D.~M. Zuckerman.
\newblock Steady-state simulations using weighted ensemble path sampling.
\newblock {\em The Journal of Chemical Physics}, 133(1):014110, July 2010.

\bibitem{cao_iterative_1985}
W.-L. Cao and W.~J. Stewart.
\newblock Iterative aggregation/disaggregation techniques for nearly uncoupled
  markov chains.
\newblock {\em Journal of the ACM}, 32(3):702--719, July 1985.

\bibitem{chatelin_acceleration_1982}
F.~Chatelin and W.~L. Miranker.
\newblock Acceleration by aggregation of successive approximation methods.
\newblock {\em Linear Algebra and its Applications}, 43:17--47, Mar. 1982.

\bibitem{darve_calculating_2001}
E.~Darve and A.~Pohorille.
\newblock Calculating free energies using average force.
\newblock {\em The Journal of Chemical Physics}, 115(20):9169--9183, Nov. 2001.

\bibitem{de_sterck_smoothed_2010}
H.~De~Sterck, T.~A. Manteuffel, S.~F. McCormick, K.~Miller, J.~Pearson,
  J.~Ruge, and G.~Sanders.
\newblock Smoothed {Aggregation} {Multigrid} for {Markov} {Chains}.
\newblock {\em SIAM Journal on Scientific Computing}, 32(1):40--61, Jan. 2010.

\bibitem{de_sterck_multilevel_2008}
H.~De~Sterck, T.~A. Manteuffel, S.~F. McCormick, Q.~Nguyen, and J.~Ruge.
\newblock Multilevel {Adaptive} {Aggregation} for {Markov} {Chains}, with
  {Application} to {Web} {Ranking}.
\newblock {\em SIAM Journal on Scientific Computing}, 30(5):2235--2262, Jan.
  2008.

\bibitem{dickson_flow-dependent_2011}
A.~Dickson, M.~Maienschein-Cline, A.~Tovo-Dwyer, J.~R. Hammond, and A.~R.
  Dinner.
\newblock Flow-{Dependent} {Unfolding} and {Refolding} of an {RNA} by
  {Nonequilibrium} {Umbrella} {Sampling}.
\newblock {\em Journal of Chemical Theory and Computation}, 7(9):2710--2720,
  Sept. 2011.

\bibitem{dinner_trajectory_2018}
A.~R. Dinner, J.~C. Mattingly, J.~O.~B. Tempkin, B.~Van~Koten, and J.~Weare.
\newblock Trajectory {Stratification} of {Stochastic} {Dynamics}.
\newblock {\em SIAM Review}, 60(4):909--938, Jan. 2018.

\bibitem{earle_convergence_2022}
G.~Earle and J.~Mattingly.
\newblock Convergence of {Stratified} {MCMC} {Sampling} of {Non}-{Reversible}
  {Dynamics}, Feb. 2022.
\newblock arXiv:2111.05838 [math].

\bibitem{geyer_markov_1991}
C.~J. Geyer.
\newblock Markov {Chain} {Monte} {Carlo} {Maximum} {Likelihood}.
\newblock {\em Interface Foundation of North America. Retrieved from the
  University of Minnesota Digital Conservancy.}, page~8, 1991.

\bibitem{golub_using_1986}
G.~H. Golub and C.~D. Meyer, Jr.
\newblock Using the {QR} {Factorization} and {Group} {Inversion} to {Compute},
  {Differentiate}, and {Estimate} the {Sensitivity} of {Stationary}
  {Probabilities} for {Markov} {Chains}.
\newblock {\em SIAM Journal on Algebraic Discrete Methods}, 7(2):273--281, Apr.
  1986.

\bibitem{haviv_aggregationdisaggregation_1987}
M.~Haviv.
\newblock Aggregation/{Disaggregation} {Methods} for {Computing} the
  {Stationary} {Distribution} of a {Markov} {Chain}.
\newblock {\em SIAM Journal on Numerical Analysis}, 24(4):952--966, Aug. 1987.

\bibitem{koury_iterative_1984}
J.~R. Koury, D.~F. McAllister, and W.~J. Stewart.
\newblock Iterative {Methods} for {Computing} {Stationary} {Distributions} of
  {Nearly} {Completely} {Decomposable} {Markov} {Chains}.
\newblock {\em SIAM Journal on Algebraic Discrete Methods}, 5(2):164--186, June
  1984.

\bibitem{krieger_two-level_1995}
U.~R. Krieger.
\newblock On a two-level multigrid solution method for finite {Markov} chains.
\newblock {\em Linear Algebra and its Applications}, 223-224:415--438, July
  1995.

\bibitem{kumar_weighted_1992}
S.~Kumar, J.~M. Rosenberg, D.~Bouzida, R.~H. Swendsen, and P.~A. Kollman.
\newblock {THE} weighted histogram analysis method for free-energy calculations
  on biomolecules. {I}. {The} method.
\newblock {\em Journal of Computational Chemistry}, 13(8):1011--1021, Oct.
  1992.

\bibitem{laio_escaping_2002}
A.~Laio and M.~Parrinello.
\newblock Escaping free-energy minima.
\newblock {\em Proceedings of the National Academy of Sciences},
  99(20):12562--12566, Oct. 2002.

\bibitem{lelievre_partial_2016}
T.~Lelièvre and G.~Stoltz.
\newblock Partial differential equations and stochastic methods in molecular
  dynamics.
\newblock {\em Acta Numerica}, 25:681--880, May 2016.

\bibitem{li_models_2009}
Y.~Li, X.~Qu, A.~Ma, G.~J. Smith, N.~F. Scherer, and A.~R. Dinner.
\newblock Models of {Single}-{Molecule} {Experiments} with {Periodic}
  {Perturbations} {Reveal} {Hidden} {Dynamics} in {RNA} {Folding}.
\newblock {\em The Journal of Physical Chemistry B}, 113(21):7579--7590, May
  2009.

\bibitem{mandel_local_1983}
J.~Mandel and B.~Sekerka.
\newblock A local convergence proof for the iterative aggregation method.
\newblock {\em Linear Algebra and its Applications}, 51:163--172, June 1983.

\bibitem{marek_convergence_1998}
I.~Marek and P.~Mayer.
\newblock Convergence analysis of an iterative aggregation/disaggregation
  method for computing stationary probability vectors of stochastic matrices.
\newblock {\em Numer. Linear Algebra Appl.}, page~22, 1998.

\bibitem{marek_convergence_2003}
I.~Marek and P.~Mayer.
\newblock Convergence theory of some classes of iterative
  aggregation/disaggregation methods for computing stationary probability
  vectors of stochastic matrices.
\newblock {\em Linear Algebra and its Applications}, 363:177--200, Apr. 2003.

\bibitem{marek_note_2006}
I.~Marek and I.~Pultarová.
\newblock A note on local and global convergence analysis of iterative
  aggregation–disaggregation methods.
\newblock {\em Linear Algebra and its Applications}, 413(2-3):327--341, Mar.
  2006.

\bibitem{marek_local_1994}
I.~Marek and D.~B. Szyld.
\newblock Local convergence of the (exact and inexact) iterative aggregation
  method for linear systems and {Markov} operators.
\newblock {\em Numerische Mathematik}, 69(1):61--82, Nov. 1994.

\bibitem{meyer_matrix_2008}
C.~D. Meyer.
\newblock {\em Matrix analysis and applied linear algebra}.
\newblock Society for Industrial and Applied Mathematics, Philadelphia, 2008.

\bibitem{norris_markov_1998}
J.~Norris.
\newblock Markov {Chains}, 1998.

\bibitem{park_reaction_2003}
S.~Park, M.~K. Sener, D.~Lu, and K.~Schulten.
\newblock Reaction paths based on mean first-passage times.
\newblock {\em The Journal of Chemical Physics}, 119(3):1313--1319, July 2003.

\bibitem{pultarova_fourier_2013}
I.~Pultarová.
\newblock Fourier {Analysis} of the {Aggregation} {Based} {Algebraic}
  {Multigrid} for {Stochastic} {Matrices}.
\newblock {\em SIAM Journal on Matrix Analysis and Applications},
  34(4):1596--1610, Jan. 2013.

\bibitem{shirts_statistically_2008}
M.~R. Shirts and J.~D. Chodera.
\newblock Statistically optimal analysis of samples from multiple equilibrium
  states.
\newblock {\em The Journal of Chemical Physics}, 129(12):124105, Sept. 2008.

\bibitem{swendsen_replica_1986}
R.~H. Swendsen and J.-S. Wang.
\newblock Replica {Monte} {Carlo} {Simulation} of {Spin}-{Glasses}.
\newblock {\em Physical Review Letters}, 57(21):2607--2609, Nov. 1986.

\bibitem{thiede_sharp_2015}
E.~Thiede, B.~Van~Koten, and J.~Weare.
\newblock Sharp {Entrywise} {Perturbation} {Bounds} for {Markov} {Chains}.
\newblock {\em SIAM Journal on Matrix Analysis and Applications},
  36(3):917--941, Jan. 2015.

\bibitem{torrie_nonphysical_1977}
G.~Torrie and J.~Valleau.
\newblock Nonphysical sampling distributions in {Monte} {Carlo} free-energy
  estimation: {Umbrella} sampling.
\newblock {\em Journal of Computational Physics}, 23(2):187--199, Feb. 1977.

\bibitem{vakhutinsky_iterative_1979}
I.~Y. Vakhutinsky, L.~M. Dudkin, and A.~A. Ryvkin.
\newblock Iterative {Aggregation}--{A} {New} {Approach} to the {Solution} of
  {Large}-{Scale} {Problems}.
\newblock {\em Econometrica}, 47(4):821, July 1979.

\bibitem{vanden-eijnden_exact_2009}
E.~Vanden-Eijnden and M.~Venturoli.
\newblock Exact rate calculations by trajectory parallelization and tilting.
\newblock {\em The Journal of Chemical Physics}, 131(4):044120, July 2009.

\bibitem{warmflash_umbrella_2007}
A.~Warmflash, P.~Bhimalapuram, and A.~R. Dinner.
\newblock Umbrella sampling for nonequilibrium processes.
\newblock {\em The Journal of Chemical Physics}, 127(15):154112, Oct. 2007.

\end{thebibliography}
\end{document}